\documentclass{llncs}

\newif\ifipmu
\ipmufalse

\usepackage{hyperref}

\usepackage{zref-clever}
\zcsetup{cap,abbrev,nameinlink=false}
\AddToHook{env/proposition/begin}{\zcsetup{countertype={theorem=proposition}}}
\AddToHook{env/corollary/begin}{\zcsetup{countertype={theorem=corollary}}}
\AddToHook{env/lemma/begin}{\zcsetup{countertype={theorem=lemma}}}

\usepackage{xcolor}

\usepackage{mathtools,amsfonts,mathrsfs}
\usepackage{enumitem}

\ifipmu
\else
\newenvironment{relproof}[1]{%
\trivlist\item[\hskip\labelsep\itshape{Proof of #1.}]
}{%
\qed
}
\fi

\DeclarePairedDelimiter{\gr}{(}{)}
\DeclarePairedDelimiter{\br}{[}{]}
\DeclarePairedDelimiterX{\brcond}[2]{[}{]}{#1\delimsize\vert#2}
\DeclarePairedDelimiter{\set}{\{}{\}}
\DeclarePairedDelimiter{\abs}{\lvert}{\rvert}

\DeclarePairedDelimiterX{\cci}[2]{[}{]}{#1,#2}
\DeclarePairedDelimiterX{\coi}[2]{[}{[}{#1,#2}
\DeclarePairedDelimiterX{\ooi}[2]{]}{[}{#1,#2}
\newcommand{\reals}{\mathbb{R}}
\newcommand{\posreals}{\mathbb{R}_{\geq0}}
\newcommand{\sposreals}{\mathbb{R}_{>0}}
\newcommand{\extreals}{\overline{\mathbb{R}}}
\newcommand{\posextreals}{\extreals\vphantom{\reals}_{\geq0}}
\newcommand{\extvars}{\overline{\mathbb{V}}}
\newcommand{\posints}{\mathbb{Z}_{\geq0}}
\newcommand{\nats}{\mathbb{N}}

\newcommand{\compl}[1]{#1^\mathrm{c}}

\newcommand{\llambda}{\underline{\lambda}}
\newcommand{\ulambda}{\overline{\lambda}}
\newcommand{\llambdai}{\llambda_\mathrm{i}}
\newcommand{\ulambdai}{\ulambda_\mathrm{i}}
\newcommand{\llambdao}{\llambda_\mathrm{o}}
\newcommand{\ulambdao}{\ulambda_\mathrm{o}}

\newcommand{\countpaths}{\Omega}
\newcommand{\pth}{\omega}
\newcommand{\vpth}{\varpi}

\DeclareMathOperator{\dom}{dom}

\newcommand{\process}[1][t]{\gr{\mathcal{S}_#1}_{#1\in\posreals}}
\newcommand{\proc}[1][t]{\mathcal{S}_{#1}}
\newcommand{\stopt}{\tau}
\newcommand{\stopts}{\mathfrak{T}}
\newcommand{\tostop}[2]{\mathcal{I}_{#1}\gr{#2}}

\newcommand{\indica}[1]{\mathbb{I}_{#1}}
\newcommand{\betstrat}{G}
\newcommand{\stake}[1]{h_{#1}}
\newcommand{\capital}[2]{\mathcal{K}^{#1}_{#2}}

\newcommand{\tscapprocs}[1][\lambda]{\mathfrak{K}_{#1}}
\newcommand{\oscapprocs}[1][\cci{\llambda}{\ulambda}]{\mathfrak{K}_{#1}}
\newcommand{\stoptdom}[1][\stopt]{\extvars_{#1}}

\newcommand{\gencapprocs}{\mathfrak{K}}
\newcommand{\gencapital}{\mathcal{K}}

\newcommand{\uprev}{\overline{\mathrm{E}}\vphantom{\mathrm{E}}}
\newcommand{\lprev}{\underline{\mathrm{E}}\vphantom{\mathrm{E}}}

\NewDocumentCommand{\gencuprev}{mO{\stopt}}{\uprev_{\gencapprocs}\br{#1\vert #2}}
\NewDocumentCommand{\genclprev}{mO{\stopt}}{\lprev_{\gencapprocs}\br{#1\vert #2}}
\newcommand{\genuprev}[1]{\uprev_{\gencapprocs}\br{#1}}

\NewDocumentCommand{\tscuprev}{O{\lambda}mO{\stopt}}{\uprev_{#1}\br{#2\vert #3}}
\NewDocumentCommand{\tsclprev}{O{\lambda}mO{\stopt}}{\lprev_{#1}\br{#2\vert #3}}
\NewDocumentCommand{\tsuprev}{O{\lambda}m}{\uprev_{#1}\br{#2}}
\NewDocumentCommand{\tslprev}{O{\lambda}m}{\lprev_{#1}\br{#2}}
\NewDocumentCommand{\oscuprev}{O{\cci{\llambda}{\ulambda}}mO{\stopt}}{\uprev_{#1}\brcond{#2}{#3}}
\NewDocumentCommand{\oscuprevstar}{O{\cci{\llambda}{\ulambda}}mO{\stopt}}{\uprev_{#1}\brcond*{#2}{#3}}
\NewDocumentCommand{\osclprev}{O{\cci{\llambda}{\ulambda}}mO{\stopt}}{\lprev_{#1}\brcond{#2}{#3}}
\NewDocumentCommand{\osuprev}{O{\cci{\llambda}{\ulambda}}m}{\uprev_{#1}\br{#2}}
\NewDocumentCommand{\oslprev}{O{\cci{\llambda}{\ulambda}}m}{\lprev_{#1}\br{#2}}

\newcommand{\bfns}[1][]{\mathbb{G}_{#1}}
\newcommand{\eye}{\mathrm{I}}
\newcommand{\poissgen}{\mathrm{G}}
\newcommand{\poisssg}{\mathrm{S}}

\newcommand{\upoissgen}{\overline{\mathrm{G}}\vphantom{\mathrm{G}}}
\newcommand{\upoisssg}{\overline{\mathrm{S}}\vphantom{\mathrm{S}}}
\newcommand{\lstake}[1]{\underline{h}_{#1}}
\newcommand{\ustake}[1]{\overline{h}_{#1}}
\newcommand{\diffstake}[1]{\overline{\underline{h}}_{#1}}

\newcommand{\utranop}{\overline{\mathrm{T}}\vphantom{\mathrm{T}}}
\newcommand{\tranop}{\mathrm{T}\vphantom{\mathrm{T}}}

\title{Allowing for Imprecision in the Game-theoretic Characterisation of the Poisson Process}
\author{Alexander Erreygers\orcidID{0000-0002-0409-2999}}
\institute{Foundations Lab for imprecise probabilities, Ghent University, Belgium}

\begin{document}

\maketitle

\begin{abstract}
In their 1993 paper `Forecasting point and continuous processes: Prequential analysis' in \emph{Test}, Vovk put forward a game-theoretic definition of the Poisson process.
A key assumption therein is that the rate of the Poisson process is known or specified exactly.
In contrast, I replace this assumption with the less stringent---and arguably more realistic---one that the available information about the process takes the form of bounds on the rate rather than a single, exact value.
The resulting process has properties similar to the standard, `precise' Poisson process, albeit with an imprecise flavour to them, thus justifying the moniker `imprecise Poisson process'.
\end{abstract}

\section{Introduction}
More than twenty years ago, Vovk~\cite{1993Vovk} put forward a `prequential' definition of the Poisson and Wiener processes,\footnote{More generally, he considers the reduction of continuous martingales to these processes through a change of time.}
with an (upper) expectation operator that is derived from the capital in a gambling game.
Nowadays, this approach towards modelling uncertainty is known as the \emph{game-theoretic} one, as laid out and popularised by Shafer \& Vovk in their seminal monographs \cite{2001ShaferVovk-Probability,2019ShaferVovk-GameTheoretic}.
Discrete-time processes have been studied extensively in this game-theoretic framework, often allowing for imprecision in the local uncertainty models.
Continuous-time processes have also received quite some attention in this framework; in contrast to discrete-time processes, I'm not aware of work in the setting of continuous time that allows for imprecision.
This is why in this contribution, I set out to allow for imprecision, or a bid--ask spread, in Vovk's~\cite{1993Vovk} game-theoretic definition of the Poisson process.

I lay down the foundation for doing so in \zcref{sec:setup}, in the form of basic notation and terminology regarding (counting) paths, variables and processes.
\zcref[S]{sec:game-theoretic upper expectations} introduces, in a relatively general manner, the basics of Shafer \& Vovk's game-theoretic foundations for modelling uncertainty regarding a continuous-time process.
Next, I specialise this to the imprecise Poisson process in \zcref{sec:betting strategies}, and investigate the properties of the resulting conditional upper expectation operator in \zcref{sec:properties}.
\zcref[S]{sec:conclusion} concludes this contribution.

\ifipmu
\else
This manuscript is accepted for publication in the proceedings of \href{https://www.sbai.uniroma1.it/conferenze/ipmu2026}{IPMU 2026}.
This extended arXiv version differs from the conference version in some parts, the biggest difference being that includes \zcref{asec:relegated-proofs}, to which I've relegated the proofs for most of the results; the sole exceptions are \zcref{prop:expected increment,prop:shifted upper expectation,the:with sublinear poisson semigroup}.
Unless mentioned otherwise, these relegated proofs are modifications of the proofs of related results in~\cite{1993Vovk}.
\fi

\section{Counting Paths, Variables \& Processes}
\label{sec:setup}
The set-up for this contribution is essentially standard, and follows Vovk's~\cite{1993Vovk} rather closely; for the reader's sake, I feel it's nonetheless necessary to take some care in introducing it clearly.
Let \(\countpaths\) be the set of all \emph{counting paths}: those paths~\(\pth\colon\posreals\to\posints\) that start in~\(0\), are increasing with \emph{unit} jumps
and right-continuous; these paths also have limits from the left, whence they form a subset of the well-known càdlàg paths \cite[Ch.~3, §~5]{1986EthierKurtz-Markov}.

An \emph{extended real partial variable} is a map from a subset of~\(\countpaths\) to the set~\(\extreals\) of extended real numbers.\footnote{
I adhere to the same extension of addition and multiplication from~\(\reals\) to~\(\extreals\) as Shafer \& Vovk~\cite[p.~420]{2019ShaferVovk-GameTheoretic}, so in particular \(+\infty-\infty=+\infty\).
}
We drop the adjective `partial' if the domain is~\(\countpaths\) and/or the adjective `extended' if the codomain is~\(\reals\).
So \(\extvars\coloneqq\extreals\vphantom{R}^\countpaths\) is the set of extended real variables, and for any time point~\(t\in\posreals\), the corresponding \emph{coordinate variable}
\begin{equation*}
N_t
\colon \countpaths\to\posints
\colon \pth\mapsto\pth(t)
\end{equation*}
is a real variable; with \(g\colon\posints\to\reals\) and \(t\in\posreals\), the composition
\begin{equation*}
    g(N_t)
    \coloneqq g\circ N_t
    \colon \countpaths\to\reals
    \colon \pth\mapsto g\gr{\pth\gr{t}}
\end{equation*}
is a real variable that is bounded if and only if \(g\) is bounded.
Indicator variables are also bounded real variables: for any \emph{event}~\(A\subseteq\countpaths\), its corresponding \emph{indicator}~\(\indica{A}\) is the \(\set{0,1}\)-valued variable that is \(1\) on~\(A\) and \(0\) elsewhere.

An essential class of extended real variables are the \emph{stopping times}: this class~\(\stopts\) consists of the positive extended real variables~\(\stopt\)---that is, maps from \(\countpaths\) to~\(\posextreals\)---such that for all \(\pth_1,\pth_2\in\countpaths\), if \(\pth_1\vert_{\cci{0}{\stopt\gr{\pth_1}}\cap\posreals}=\pth_2\vert_{\cci{0}{\stopt\gr{\pth_1}}\cap\posreals}\) then \(\stopt\gr{\pth_1}=\stopt\gr{\pth_2}\).
Note that the constant real variable \(t\colon\pth\mapsto t\) is a stopping time, and that for all stopping times~\(\stopt_1, \stopt_2\in\stopts\), their pointwise minimum~\(\stopt_1\wedge\stopt_2\) and maximum~\(\stopt_1\vee\stopt_2\) are also stopping times.
For any stopping time~\(\stopt\in\stopts\) and path~\(\pth\), we let \(\tostop{\stopt}{\pth}\) be the set of counting paths that agree with \(\pth\) on~\(\cci{0}{\stopt\gr{\pth}}\cap\posreals\):
\begin{equation*}
\tostop{\stopt}{\pth}
\coloneqq \set{\vpth\in\countpaths\colon \vpth\vert_{\cci{0}{\stopt\gr{\pth}}\cap\posreals}=\pth\vert_{\cci{0}{\stopt\gr{\pth}}\cap\posreals}}.
\end{equation*}
An extended real partial variable~\(f\) is \emph{\(\stopt\)-measurable} if for all \(\pth_1,\pth_2\in\dom f\), \[\pth_1\vert_{\cci{0}{\stopt\gr{\pth_1}}\cap\posreals}=\pth_2\vert_{\cci{0}{\stopt\gr{\pth_1}}\cap\posreals}\implies f(\pth_1)=f(\pth_2);\] differently put, \(f\) is constant on~\(\tostop{\stopt}{\pth}\) for all \(\pth\in\dom f\).
Although a stopping time~\(\stopt\) can take the value \(+\infty\), we are typically interested in the set of paths for which it is finite/real; we denote this set by \(\set{\stopt<+\infty}\coloneqq \set{\pth\in\countpaths\colon \stopt\gr{\pth}<+\infty}\).
For example, for any stopping time~\(\stopt\), we denote the set of extended real partial variables whose domain includes~\(\set{\stopt<+\infty}\) by~\(\stoptdom\), and for any \(\pth\in\set{\stopt<+\infty}\), we shorten \(\pth\gr{\stopt\gr{\pth}}\) to~\(\pth\gr{\stopt}\).

Finally, a \emph{process} is a family of real variables~\(\proc[\bullet]=\process\) such that \(\proc\) is \(t\)-measurable for all \(t\in\posreals\).
Such a process~\(\proc[\bullet]\) is \emph{bounded below} if
\begin{equation*}
\inf\proc[\bullet]
\coloneqq\inf\set{\proc[t]\gr{\pth}\colon t\in\posreals, \pth\in\countpaths}
> -\infty.
\end{equation*}
The canonical example of a process that is bounded below is the \emph{coordinate process}~\(N_\bullet\).
For any process~\(\proc[\bullet]\) and stopping time~\(\stopt\), we'll often consider the `stopped process'
\(
\proc[\stopt]
\colon \set{\stopt<+\infty}\to\reals
\colon \pth\mapsto \proc[\stopt\gr{\pth}]\gr{\pth\gr{\stopt}}
\)
and the extended real variable
\begin{equation*}
\liminf \proc[\bullet]
\colon \countpaths\to\extreals
\colon \pth\mapsto
\liminf_{t\to+\infty} \proc[t]\gr{\pth};
\end{equation*}
note that if \(\proc[\bullet]\) is bounded below, then so is \(\liminf\proc[\bullet]\).


\section{Game-theoretic Upper Expectations}\label{sec:game-theoretic upper expectations}
In this contribution, we'll use Shafer \& Vovk's game-theoretic approach to model the uncertain (future) evolution of the coordinate process~\(N_\bullet\).
A lot can be said about their powerful game-theoretic framework, but we'll stick to what is necessary for the remainder of this contribution; we refer the interested reader to their monographs~\cite{2001ShaferVovk-Probability,2019ShaferVovk-GameTheoretic} and references therein for more background.
In the setting of continuous time, Shafer \& Vovk consider a game between two players, called \emph{trader} and \emph{market}: first trader announces their strategy to trade in a number of securities, then market determines the price path of the securities.
As in discrete time, one can also imagine a third player, called \emph{forecaster}: they choose the securities that are available to trader.
In the context of this contribution, the securities are derived from the coordinate process~\(N_\bullet\) and market chooses one realisation~\(\pth\) from~\(\countpaths\).

Trader's strategy has to obey some restrictions for this game to make sense.
For example, it makes sense to assume that at time~\(t\in\posreals\), they are uncertain about the future prices of the securities, or differently put, only have exact information about the prices of the securities in the past.
We'll get back to these restrictions in \zcref{sec:betting strategies} further on.
For now, we'll assume that each allowed trading strategy gives rise to a process \(\gencapital_\bullet\), which we'll call trader's \emph{capital process} under this strategy; recall from \zcref{sec:setup} that for \(\gencapital_\bullet\) to be a process, it must be that \(\gencapital_t\) is \(t\)-measurable for all \(t\in\posreals\), meaning that it's completely defined by the value of the paths on~\(\cci{0}{t}\)---or, in other words, doesn't depend on the future values of the prices of the securities.

The set~\(\gencapprocs\) collects the capital processes corresponding to all allowed trading strategies.
Then the conditional upper expectation~\(\gencuprev{\bullet}\colon \stoptdom\times\set{\stopt<+\infty}\to\extreals\) is defined for all \(f\in\stoptdom\) and \(\pth\in\set{\stopt<+\infty}\) by
\begin{equation*}
\gencuprev{f}\gr{\pth}
\coloneqq \inf \set[\big]{\gencapital_\stopt\gr{\pth}\colon \gencapital_\bullet\in\gencapprocs, \inf\gencapital_\bullet>-\infty, \liminf\gencapital_\bullet\geq_{\tostop{\stopt}{\pth}} f},
\end{equation*}
where here and in the remainder, we write~`\(\liminf\gencapital_\bullet\geq_{\tostop{\stopt}{\pth}} f\)' as a shorthand for `\(\liminf \gencapital_\bullet\gr{\vpth}\geq f\gr{\vpth}\) for all \(\vpth\in\tostop{\stopt}{\pth}\)'.
This conditional expectation~\(\gencuprev{f}\gr{\pth}\) is the infimum capital trader needs at time~\(\stopt\gr{\pth}\) to (`in the long run') superhedge \(f\) in `all possible futures'~\(\tostop{\stopt}{\pth}\) using an allowed trading strategy for which the losses are bounded below a priori.
Differently put, if trader has capital~\(c>\gencuprev{f}\gr{\pth}\) at~\(\stopt\gr{\pth}\) (and knowing the partial counting path~\(\pth\vert_{\cci{0}{\stopt\gr{\pth}}}\)), they can trade according to an allowed trading strategy such that they'll have~\(f\gr{\vpth}\) in the long run, whatever the actual realisation \(\vpth\in\tostop{\stopt}{\pth}\) turns out to be.
Consequently, we can interpret \(\gencuprev{f}\gr{\pth}\) as trader's infimum selling price for the uncertain pay-off~\(f\), given the information they have at time~\(\stopt\gr{\pth}\).
Since \(\pth\gr{0}=0\) for all \(\pth\in\countpaths\), \(\gencapital_0\) is constant on~\(\countpaths\), and the same is true for~\(\gencuprev{f}[0]\), which for this reason we'll shorten to~\(\genuprev{f}\) from now on; this can be interpreted as the infimum initial capital trader needs at time~\(0\) to superhedge~\(f\) `in the long run', or alternatively, their infimum selling price for the uncertain pay-off~\(f\).

For every stopping time~\(\stopt\in\stopts\) and extended variable~\(f\in\stoptdom\), our definition ensures that the extended real partial variable~\(\pth\mapsto\gencuprev{f}\gr{\pth}\) is constant on~\(\tostop{\stopt}{\pth}\), and therefore \(\stopt\)-measurable.
Furthermore, for every extended variable~\(f\in\extvars\) and path~\(\pth\in\countpaths\), the function \(t\mapsto\gencuprev{f}[t]\gr{\pth}\) is increasing because the map~\(t\mapsto\tostop{t}{\pth}\) is decreasing.

Our calling~\(\gencuprev{\bullet}[\bullet]\) a (conditional) upper expectation---after \cite{1975Williams}---is justified under some mild assumptions on~\(\gencapprocs\).
\begin{proposition}\label{prop:gencuprev}
Suppose \(\gencapprocs\) contains all constant processes and is a cone---that is, closed under pointwise addition and multiplication with positive scalars.
Then for all \(\stopt\in\stopts\), \(f, g\in\stoptdom\)  and \(\mu\in\reals\),
\begin{enumerate}[label=\upshape{E\arabic*.}, ref=\upshape(E\arabic*), leftmargin=*, series=gencuprev]
\item\label{gencuprev:sup} \(\gencuprev{f}\leq\sup f\);
\item\label{gencuprev:subadditivity} \(\gencuprev{f+g}\leq\gencuprev{f}+\gencuprev{g}\);
\item\label{gencuprev:positive homogeneity} \(\gencuprev{\mu f}=\mu \gencuprev{f}\) whenever \(\mu>0\);
\item\label{gencuprev:monotonicity} \(\gencuprev{f}\leq\gencuprev{g}\) whenever \(f\leq g\);
\item\label{gencuprev:constant additivity} \(\gencuprev{f+\mu}=\gencuprev{f}+\mu\).
\end{enumerate}
\end{proposition}

Following Vovk~\cite[p.~196]{1993Vovk}, we call the set~\(\gencapprocs\) of allowed capital processes \emph{coherent} if
\begin{equation}\label{eq:coherence condition}
\inf\set*{\liminf \gencapital_\bullet\gr{\vpth}\colon \vpth\in\tostop{t}{\pth}}
\leq \gencapital_t\gr{\pth}
\quad\text{for all } \gencapital_\bullet\in\gencapprocs, t\in\posreals, \pth\in\countpaths;
\end{equation}
in other words, this coherence condition ensures that no trading strategy can result in a guaranteed profit for trader.
That we demand coherence should come to no surprise to the reader who is familiar with coherent (conditional) upper expectations a la de Finetti~\cite{2017Finetti-Theory}, Williams~\cite{1975Williams} and/or Walley~\cite{1991Walley}---see also \cite{2014TroffaesDeCooman-Lower}.
As is customary in that setting, we define the conditional lower expectation through conjugacy:
\begin{equation*}
\genclprev{f}\gr{\pth}
\coloneqq -\gencuprev{-f}\gr{\pth}
\quad\text{for all } \stopt\in\stopts, f\in\stoptdom, \pth\in\set{\stopt<+\infty};
\end{equation*}
the value~\(\genclprev{f}\gr{\pth}\) can be thought of as trader's supremum buying price for~\(f\) given the information they have at~\(\stopt\gr{\pth}\).
If \(\gencapprocs\) is coherent, trader's buying price is never higher than their selling price.
\begin{proposition}\label{prop:gencuprev with coherence}
If \(\gencapprocs\) satisfies the conditions in \zcref{prop:gencuprev} and is coherent, then for all \(\stopt\in\stopts\), \(f\in\stoptdom\) and \(\pth\in\set{\stopt<+\infty}\),
\begin{enumerate}[resume*=gencuprev]
\item\label{gencuprev:better bounds} \(\inf f\vert_{\tostop{\stopt}{\pth}}\leq \gencuprev{f}\gr{\pth}\leq \sup f\vert_{\tostop{\stopt}{\pth}}\);
\item\label{gencuprev:lower upper} \(\genclprev{f}\gr{\pth}\leq\gencuprev{f}\gr{\pth}\).
\end{enumerate}
\end{proposition}

\section{Betting Strategies \& Capital Processes for the Poisson Process}\label{sec:betting strategies}
As announced in the Introduction, this contribution aims to extend Vovk's~\cite{1993Vovk} game-theoretic description of the Poisson process to allow for imprecision.
Their game-theoretic characterisation draws inspiration from Watanabe's martingale characterisation of the Poisson process \cite[Theorem~2.3 and following Remark]{1964Watanabe}: the coordinate process~\(N_\bullet\) is a Poisson process with rate~\(\lambda\in\posreals\) if and only if the compensated process~\(\gr{N_t-\lambda t}_{t\in\posreals}\) is a martingale---that is, if conditional on the information up to the current time point~\(s\in\posreals\), for every future time point~\(t\in\ooi{s}{+\infty}\) the expectation of~\(N_t-\lambda t\) is~\(N_s-\lambda s\), or differently put, the expectation of the increment~\(N_t-N_s\) is~\(\lambda\gr{t-s}\).
Vovk's allowed trading strategies are exactly those that involve this compensated increment.

If forecaster puts forth some \emph{rate}~\(\lambda\in\posreals\), trader is allowed to bet on a series of compensated increments of the form~\(N_{\stopt_{k+1}}-N_{\stopt_k}-\lambda\gr{\stopt_{k+1}-\stopt_k}\), for stopping times~\(\stopt_k\leq\stopt_{k+1}\).
More formally, a \emph{two-sided elementary trading strategy}~\(\betstrat\) is an increasing sequence of stopping times \(\stopt_1\leq\stopt_2\leq\cdots\leq \stopt_n\leq\stopt_{n+1}\) and a sequence~\(\stake{1}, \dots, \stake{n}\) of variables such that for all \(k\in\set{1, \dots, n}\), the \emph{stake}~\(\stake{k}\) is a (possibly partial) bounded real variable with \(\dom\stake{k}\supseteq\set{\stopt_k<+\infty}\) that is \(\stopt_k\)-measurable.
We interpret such an elementary trading strategy~\(\betstrat\) as follows: at time~\(\stopt_k\), trader puts stake~\(h_k\) on the increment~\(N_{\stopt_{k+1}}-N_{\stopt_k}-\lambda\gr{\stopt_{k+1}-\stopt_k}\)---or more accurately, they `pay' \(\stake{k}\gr{\pth}\gr{\pth\gr{\stopt_k}-\lambda\stopt_{k}\gr{\pth}}\) to receive \(\stake{k}\gr{\pth}\gr{\pth\gr{\stopt_{k+1}}-\lambda\stopt_{k+1}\gr{\pth}}\) at time~\(\stopt_{k+1}\) of their choice.
Consequently, if trader starts with initial capital~\(c\in\reals\) and follows the two-sided elementary trading strategy~\(\betstrat\), their capital at time~\(t\in\posreals\) is the real variable
\begin{equation}\label{eqn:simple capital process}
\capital{\betstrat, c}{t}
\coloneqq c+\sum_{k=1}^n h_k\cdot \gr[\big]{N_{\stopt_{k+1}\wedge t}-N_{\stopt_k\wedge t} - \lambda\gr{\stopt_{k+1}\wedge t-\stopt_k\wedge t}},
\end{equation}
where \(\cdot\) indicates pointwise multiplication of two real variables and where we ignore the zero terms in the sum that arise when \(\stopt_k\gr{\pth}\wedge t=\stopt_{k+1}\gr{\pth}\wedge t\).
Our assumptions on the stopping times~\(\stopt_k\) and stakes~\(h_k\) ensure that \(\gr{\capital{\betstrat, c}{t}}_{t\in\posreals}\) is a real process, which we call a \emph{two-sided elementary capital process}.
We collect these capital processes in the set~\(\tscapprocs\), and denote the corresponding conditional upper and lower expectation by~\(\tscuprev{\bullet}[\bullet]\) and \(\tsclprev{\bullet}[\bullet]\), respectively.
It's easy to see that \(\tscapprocs\) includes the constant processes and is closed under scalar multiplication, but it takes a bit more work to verify that \(\tscapprocs\) is coherent and closed under pointwise addition; we don't prove this here separately, since it's a special case of \zcref{prop:oscaprocss:co co co} further on.

In the two-sided betting strategies, trader bets on the increment~\(N_{\stopt_{k+1}}-N_{\stopt_k}\) being `large'---greater than \(\lambda\gr{\stopt_{k+1}-\stopt_k}\)---when \(\stake{k}\geq0\), and on \(N_{\stopt_{k+1}}-N_{\stopt_k}\) being `small'---smaller than \(\lambda\gr{\stopt_{k+1}-\stopt_k}\)---when \(\stake{k}\leq0\).
In the one-sided betting strategies that we're going introduce next, there's a spread between the prices for these two bets.
This time around, forecaster puts forth \emph{rate bounds}~\(\llambda, \ulambda\in\posreals\) such that \(\llambda\leq\ulambda\), and trader can take a \emph{positive} stake on a series of bets~\(N_{\stopt_{k+1}}-N_{\stopt_k}-\ulambda\gr{\stopt_{k+1}-\stopt_k}\) and \(\llambda\gr{\stopt_{k+1}-\stopt_k}-\gr{N_{\stopt_{k+1}}-N_{\stopt_k}}\).
More formally, a \emph{one-sided elementary trading strategy}~\(\betstrat\) is an increasing sequence of stopping times \(\stopt_1\leq\stopt_2\leq\cdots\leq \stopt_n\leq\stopt_{n+1}\) and a sequence~\(\ustake{1}, \lstake{1}, \dots, \ustake{n}, \lstake{n}\) of variables such that for all \(k\in\set{1, \dots, n}\), the \emph{stakes}~\(\ustake{k}\) and \(\lstake{k}\) are \(\stopt_k\)-measurable (possibly partial) bounded \emph{positive} real variables whose domain includes \(\set{\stopt_k<+\infty}\).
If trader starts with initial capital~\(c\in\reals\) and follows the one-sided elementary trading strategy~\(\betstrat\), their corresponding capital at time~\(t\in\posreals\) is the real variable
\begin{multline}\label{eqn:one-sided simple capital process}
\capital{\betstrat, c}{t}
\coloneqq c+\sum_{k=1}^n \ustake{k}\cdot \gr[\big]{N_{\stopt_{k+1}\wedge t}-N_{\stopt_k\wedge t} - \ulambda\gr{\stopt_{k+1}\wedge t-\stopt_k\wedge t}}\\+\lstake{k}\cdot \gr[\big]{\llambda\gr{\stopt_{k+1}\wedge t-\stopt_k\wedge t}-N_{\stopt_{k+1}\wedge t}+N_{\stopt_k\wedge t}}.
\end{multline}
Our assumptions on the stopping times~\(\stopt_k\) and stakes~\(\ustake{k}\) and \(\lstake{k}\) ensure that \(\gr{\capital{\betstrat, c}{t}}_{t\in\posreals}\) is a real process, which we call a \emph{one-sided elementary capital process}.
We collect these capital processes in the set~\(\oscapprocs\), and denote the corresponding conditional upper and lower expectations by~\(\oscuprev{\bullet}[\bullet]\) and \(\osclprev{\bullet}[\bullet]\), respectively.

For every capital process~\(\capital{\betstrat, c}{\bullet}\in\oscapprocs\), index~\(k\in\set{1, \dots, n}\), counting path~\(\pth\in\countpaths\) and time points~\(t,r\in\cci{\stopt_k\gr{\pth}}{\stopt_{k+1}\gr{\pth}}\cap\posreals\) such that \(t\leq r\), it follows from \zcref{eqn:one-sided simple capital process} that, with \(\diffstake{k}\coloneqq\ustake{k}-\lstake{k}\),
\begin{equation}\label{eqn:one-sided capital process different increment}
\capital{\betstrat, c}{r}\gr{\pth} - \capital{\betstrat, c}{t}\gr{\pth}
= \diffstake{k}\gr{\pth}\gr[\big]{\pth\gr{r}-\pth\gr{t}-\llambda\gr{r-t}}\\-\ustake{k}\gr{\pth}\gr{\ulambda-\llambda}\gr{r-t}.
\end{equation}
This equality makes it obvious that for all \(\lambda\in\posreals\), the set~\(\oscapprocs[\set{\lambda}]\) of one-sided elementary capital processes for~\(\cci{\lambda}{\lambda}=\set{\lambda}\) is equal to the set~\(\tscapprocs\) of two-sided elementary capital processes for~\(\lambda\).

Recall from \zcref{sec:game-theoretic upper expectations} that it's crucial that the set of capital processes~\(\oscapprocs\) satisfies the conditions in \zcref{prop:gencuprev with coherence}.
\begin{proposition}\label{prop:oscaprocss:co co co}
For all \(\llambda, \ulambda\in\posreals\) with \(\llambda\leq\ulambda\), the set~\(\oscapprocs\) of one-sided elementary capital processes includes the constants, is a cone and is coherent.
\end{proposition}

\section{Properties of the Conditional Upper Expectation for the Poisson Process}\label{sec:properties}
In the remainder of this contribution, I list some properties of the conditional upper expectation operator~\(\oscuprev{\bullet}[\bullet]\) which I hope should convince the reader of the sensibility of the game-theoretic characterisation of the imprecise Poisson process.

\subsection{Properties Related to the Rate Bounds}
The first property looks at nested rate bounds: the conditional upper expectation for~\(\cci{\llambdai}{\ulambdai}\) is dominated by the one for~\(\cci{\llambdao}{\ulambdao}\) whenever \(\cci{\llambdai}{\ulambdai}\subseteq\cci{\llambdao}{\ulambdao}\).
\begin{proposition}\label{prop:monotonicity in lambda}
For all \(\llambdao,\llambdai, \ulambdai, \ulambdao\in\posreals\) with \(\llambdao\leq\llambdai\leq\ulambdai\leq\ulambdao\),
\begin{equation*}
\oscuprev[\cci{\llambdai}{\ulambdai}]{f}
\leq \oscuprev[\cci{\llambdao}{\ulambdao}]{f}
\quad\text{for all } \stopt\in\stopts, f\in\stoptdom.
\end{equation*}
\end{proposition}

The second property involves the interpretation of the rate bounds~\(\ulambda\) and \(\llambda\) as proportionality constants for forecaster's selling and buying prices.
Since, as explained in \zcref{sec:game-theoretic upper expectations}, \(\oscuprev{N_\stopt-N_\sigma}[\sigma]\gr{\pth}\) can be interpreted as trader's infimum selling price for the uncertain increment~\(N_\stopt-N_\sigma\) given the information they have at time~\(\sigma\gr{\pth}\), one would expect that like forecaster's price, this is proportional to~\(\ulambda\) ; the next result confirms that this is indeed the case, at least when these stopping times are \emph{constant}.
\begin{proposition}\label{prop:expected increment}
For any two time points~\(s, t\in\posreals\) such that \(s\leq t\),
\begin{equation*}
\oscuprev{N_t-N_s}[s]
= \ulambda\gr{t-s}
\quad\text{and}\quad
\osclprev{N_t-N_s}[s]
= \llambda\gr{t-s}.
\end{equation*}
\end{proposition}
\begin{proof}
We prove first that
\begin{equation}\label{eq:bounds on expected increment}
\oscuprev{N_t-N_s}[s]
\leq \ulambda\gr{t-s}
\quad\text{and}\quad
\oscuprev{-\gr{N_t-N_s}}[s]
\leq -\llambda\gr{t-s}.
\end{equation}

Fix any \(m\in\nats\), let \(c_m\coloneqq\ulambda\gr{t-s}/m\) and observe that for the trading strategy~\(\betstrat_m\) with stopping times~\(\stopt_1\coloneqq s\) and \(\stopt_2\coloneqq t\) and stakes~\(\ustake{1}\coloneqq 1/m\) and \(\lstake{1}\coloneqq 0\),
\begin{equation*}
\capital{\betstrat_m, c_m}{r}
= \begin{cases}
\frac1m\ulambda\gr{t-s} &\text{if } r\leq  s \\
\frac1m\gr[\big]{\ulambda\gr{t-r}+N_r-N_s} &\text{if } s<r<t \\
\frac1m\gr{N_t-N_s} &\text{if } r\geq t
\end{cases}
\quad\text{for all } r\in\posreals.
\end{equation*}
Since \(\capital{\betstrat_m, c_m}{\bullet}\) is non-negative and superhedges~\(\gr{N_t-N_s}/m\), we've shown that
\begin{equation*}
\oscuprev{\gr{N_t-N_s}/m}[s]
\leq \frac1m\ulambda\gr{t-s}.
\end{equation*}
Setting \(m=1\) gives the first inequality in~\eqref{eq:bounds on expected increment}.

Proving the second inequality is a bit more involved.
Let \(c\coloneqq-\llambda\gr{t-s}\).
For all \(m\in\nats\), consider the trading strategy~\(\betstrat'_m\) with \(\stopt^m_1=s\), \(\ustake{1}^m\coloneqq0\), \(\lstake{1}^m\coloneqq1\) and
\begin{equation*}
\stopt^m_2
\colon \countpaths\to\posextreals
\colon \pth\mapsto \inf\set{r\in\coi{s}{+\infty}\colon \pth\gr{r}\geq\pth\gr{s}+m}\cup\set{t}.
\end{equation*}
Then for all \(\pth\in\countpaths\) and \(r\in\posreals\),
\begin{equation*}
\capital{\betstrat'_m, c}{r}\gr{\pth}
= \begin{cases}
-\llambda\gr{t-s} &\text{if } r\leq s \\
-\llambda\gr{t-r} -\gr[\big]{\pth\gr{r}-\pth\gr{s}} &\text{if } s<r<\stopt^m_2\gr{\pth} \\
-\gr[\big]{\pth\gr{t}-\pth\gr{s}} &\text{if } r\geq \stopt^m_2\gr{\pth}=t \\
-\llambda\gr{t-\stopt^m_2\gr{\pth}} - m &\text{if  } r\geq \stopt^m_2\gr{\pth}<t.
\end{cases}
\end{equation*}
This equality tells us that \(\capital{\betstrat'_m, c}{\bullet}\) is bounded below and superhedges~\(-\gr{N_t-N_s}\) on~\(\set{\pth\in\countpaths\colon \pth\gr{t}<\pth\gr{s}+m}\).
Since \(\capital{\betstrat_m, c_m}{\bullet}\) is bounded below and superhedges~\(\gr{N_t-N_s}/m\), it follows that \(\capital{\betstrat'_m, c}{\bullet}+\llambda\gr{t-s}\capital{\betstrat_m, c_m}{\bullet}\) is bounded below and superhedges~\(-\gr{N_t-N_s}\).
Consequently,
\begin{equation*}
\oscuprev{-\gr{N_t-N_s}}[s]
\leq \capital{\betstrat'_m, c}{s}+\llambda\gr{t-s}\capital{\betstrat_m, c_m}{s}
= -\llambda\gr{t-s}+\frac1m\llambda\ulambda\gr{t-s}^2.
\end{equation*}
Since this inequality holds for arbitrary~\(m\in\nats\), it implies the second inequality in~\eqref{eq:bounds on expected increment}.

If \(\llambda=\lambda=\ulambda\), we infer from the two inequalities in \eqref{eq:bounds on expected increment} that
\begin{equation*}
\lambda\gr{t-s}
\leq
\tsclprev{N_t-N_s}[s]
\leq \tscuprev{N_t-N_s}[s]
\leq \lambda\gr{t-s}.
\end{equation*}
If \(\llambda<\ulambda\), the equality follows from \eqref{eq:bounds on expected increment} and the equality above for \(\lambda=\ulambda\) and \(\lambda=\llambda\) due to \zcref{prop:monotonicity in lambda} [with \(\llambdai=\lambda=\ulambdai\), \(\llambdao=\llambda\) and \(\ulambdao=\ulambda\)].
\qed\end{proof}

The next property we turn to is the alternative characterisation of the rate parameter~\(\lambda\) of the `classical' Poisson process as the inverse of the expected time until the next increment.
The time until the next increment after the stopping time~\(\stopt\in\stopts\) is the partial variable
\begin{equation*}
\rho_\stopt
\colon \set{\stopt<+\infty}\to\posextreals\colon
\pth\mapsto \inf\set{r\in\ooi{\stopt\gr{\pth}}{+\infty}\colon \pth\gr{r}>\pth\gr{\stopt}};
\end{equation*}
note that we can extend the domain of~\(\rho_\stopt\) to~\(\countpaths\) by setting it equal to~\(+\infty\) on~\(\set{\stopt<+\infty}^\mathrm{c}\), which makes this variable a stopping time.
It's fairly easy to verify that the upper and lower expectation of~\(\rho_\stopt-\stopt\) conditional on~\(\stopt\) are the inverse of the forecaster's rate bounds.
\begin{proposition}\label{prop:expected renewal time}
For every stopping time~\(\stopt\in\stopts\),
\begin{equation*}
\oscuprev{\rho_\stopt-\stopt}
= \frac1{\llambda}
\quad\text{and}\quad
\osclprev{\rho_\stopt-\stopt}
= \frac1{\ulambda},
\end{equation*}
where here we follow the convention~\(1/0=+\infty\).
\end{proposition}
\ifipmu
\else
Since the proof for this result is similar to but a bit more involved as the one for \zcref{prop:expected increment}, it's been relegated to \zcref{asec:proof of expected renewal time}.
\fi

\subsection{The Strong Markov Property}
The next property of the `classical' Poisson process on the list is that it satisfies the (strong) Markov property.
We'll actually show a stronger `memorylessness'-like result first: that the conditional upper expectation~\(\oscuprev{\bullet}\) can be `shifted' from the stopping time~\(\stopt\) to \(0\)---intuitively, this is a straightforward consequence of the fact that trading strategies can be `shifted'.
In the formal statement of this result, we need to stitch together two counting paths; following \cite{2017Neufeld}, we stitch together paths~\(\pth,\vpth\in\countpaths\) at the stopping time~\(\stopt\in\stopts\) as follows:
\begin{equation*}
\pth\oplus_\stopt\vpth
\colon \posreals\to\posints
\colon t\mapsto \begin{cases}
\pth\gr{t} & \text{if } t<\stopt\gr{\pth},\\
\pth\gr{\stopt\gr{\pth}} + \vpth\gr{t-\stopt\gr{\pth}} & \text{if } t\geq \stopt\gr{\pth};
\end{cases}
\end{equation*}
the reader will have no trouble verifying that \(\pth\oplus_\stopt\vpth\) is indeed a counting path.
Similarly, for any time point~\(s\in\posreals\) and counting path~\(\pth\in\countpaths\), we'll need to look at the derived counting path
\begin{equation*}
\pth_{\coi{s}{+\infty}}
\colon \posreals\to\posints
\colon t\mapsto \pth\gr{t+s}-\pth\gr{s}
\end{equation*}
which looks only at the increments of~\(\pth\) starting from~\(s\).
\begin{proposition}\label{prop:shifted upper expectation}
For all \(\stopt\in\stopts\), \(f\in\stoptdom\) and \(\pth\in\set{\stopt<+\infty}\),
\begin{equation*}
\oscuprev{f}\gr{\pth}
= \osuprev{f\gr{\pth\oplus_\stopt \bullet}}
\quad\text{where }
f\gr{\pth\oplus_\stopt \bullet}
\colon \countpaths\to\extreals
\colon \vpth \mapsto f\gr{\pth\oplus_\stopt \vpth}.
\end{equation*}
\end{proposition}
\begin{proof}
We'll verify the equality in the statement by proving that the left-hand side of the equality is lower than or equal to the right-hand side and vice versa.

First, we show that the left-hand side of the equality is lower than or equal to the right-hand side.
Consider any bounded below capital process~\(\capital{\betstrat, c}{\bullet}\in\oscapprocs\) that superhedges \(f\gr{\pth\oplus\bullet}\).
Let us enumerate the stopping times and stakes of~\(\betstrat\) as \(\stopt_1\), \dots, \(\stopt_{n+1}\) and \(\ustake{1}\), \dots, \(\lstake{n+1}\), respectively.
Consider the trading strategy~\(\betstrat'\) with stopping times \(\stopt_1'\), \dots, \(\stopt'_{n+1}\) defined for all \(k\in\set{1, \dots, n+1}\) as
\begin{equation*}
\stopt'_k
\colon \countpaths\to\posextreals
\colon \vpth\mapsto \begin{cases}
\stopt\gr{\pth}+\stopt_k\gr{\vpth_{\coi{\stopt\gr{\pth}}{+\infty}}} &\text{if } \vpth\in\tostop{\stopt}{\pth} \\
+\infty & \text{if } \vpth\notin\tostop{\stopt}{\pth}
\end{cases}
\end{equation*}
and, for all \(k\in\set{1, \dots, n}\), stakes defined on~\(\tostop{\stopt}{\pth}\) by
\begin{equation*}
\ustake{k}'\gr{\vpth}
\coloneqq \ustake{k}\gr{\vpth_{\coi{\stopt\gr{\pth}}{+\infty}}}
\quad\text{and}\quad
\lstake{k}'\gr{\vpth}
\coloneqq \lstake{k}\gr{\vpth_{\coi{\stopt\gr{\pth}}{+\infty}}}
\quad\text{for all } \vpth\in\tostop{\stopt}{\pth}.
\end{equation*}
Our construction ensures that \(\capital{\betstrat', c}{\bullet}\in\oscapprocs\) is bounded below with \(\capital{\betstrat', c}{\stopt}\gr{\pth}=c\) and
\begin{equation*}
\capital{\betstrat', c}{r}\gr{\vpth}
= \capital{\betstrat, c}{r-\stopt\gr{\pth}}\gr{\vpth_{\coi{\stopt\gr{\pth}}{+\infty}}}
\quad\text{for all } \vpth\in\tostop{\stopt}{\pth}, r\in\coi{\stopt\gr{\pth}}{+\infty}.
\end{equation*}
For all \(\vpth\in\tostop{\stopt}{\pth}\), we infer from all this that
\begin{equation*}
\liminf\capital{\betstrat', c}{\bullet}\gr{\vpth}
= \liminf\capital{\betstrat, c}{\bullet}\gr{\vpth_{\coi{\stopt\gr{\pth}}{+\infty}}}
\geq f\gr{\pth\oplus_\stopt\vpth_{\coi{\stopt\gr{\pth}}{+\infty}}}
= f\gr{\vpth},
\end{equation*}
where for the last equality we used that \(\vpth=\pth\oplus_\stopt\vpth_{\coi{\stopt\gr{\pth}}{+\infty}}\).
So for any bounded below capital process~\(\capital{\betstrat, c}{\bullet}\in\oscapprocs\) that superhedges~\(f\gr{\pth\oplus_\stopt\bullet}\), there is some bounded below capital process~\(\capital{\betstrat', c}{\bullet}\in\oscapprocs\) with \(\capital{\betstrat'c}{\stopt}\gr{\pth}=c\) that superhedges~\(f\) on~\(\tostop{\stopt}{\pth}\).
From this, we infer that indeed
\begin{equation*}
\oscuprev{f}\gr{\pth}
\leq \osuprev{f\gr{\pth\oplus_\stopt \bullet}}.
\end{equation*}

\ifipmu
\else
The proof of the reverse inequality is similar.
Consider any bounded below capital process~\(\capital{\betstrat, c}{\bullet}\in\oscapprocs\) that superhedges~\(f\) on~\(\tostop{\stopt}{\pth}\) with stopping times \(\stopt_1\), \dots, \(\stopt_{n+1}\) and stakes~\(\ustake{1}\), \dots, \(\lstake{n+1}\).
We now consider the trading strategy~\(\betstrat'\) with stopping times given for all \(k\in\set{1, \dots, n+1}\) by
\begin{equation*}
\stopt'_k
\colon \countpaths\to\posextreals
\colon \vpth\mapsto \begin{cases}
\stopt_k\gr{\pth\oplus_\stopt\vpth}-\stopt\gr{\pth} &\text{if } \stopt_k\gr{\pth\oplus_\stopt\vpth}\geq\stopt\gr{\pth} \\
0 &\text{otherwise}
\end{cases}
\end{equation*}
and stakes defined, for all \(k\in\set{1, \dots, n}\), by
\begin{equation*}
\ustake{k}'\gr{\vpth}
\coloneqq \ustake{k}\gr{\pth\oplus_\stopt\vpth}
\quad\text{and}\quad
\lstake{k}'\gr{\vpth}
\coloneqq \lstake{k}\gr{\pth\oplus_\stopt\vpth}
\quad\text{for all } \vpth\in\countpaths.
\end{equation*}
With \(c'\coloneqq\capital{\betstrat, c}{\stopt}\gr{\pth}\), our construction ensures that
\begin{equation*}
\capital{\betstrat', c'}{r}\gr{\vpth}
= \capital{\betstrat, c}{\stopt\gr{\pth}+r}\gr{\pth\oplus_\stopt\vpth}
\quad\text{for all } \vpth\in\countpaths, r\in\posreals,
\end{equation*}
and therefore \(\capital{\betstrat', c'}{\bullet}\in\oscapprocs\) is bounded below with
\begin{equation*}
\liminf \capital{\betstrat', c'}{\bullet}\gr{\vpth}
= \liminf\capital{\betstrat, c}{\bullet}\gr{\pth\oplus_\stopt\vpth}
\geq f\gr{\pth\oplus_\stopt\vpth}
\quad\text{for all } \vpth\in\countpaths.
\end{equation*}
Since \(\tostop{\stopt}{\pth}=\set{\pth\oplus_\stopt\vpth\colon \vpth\in\countpaths}\), we've shown that for any bounded below capital process~\(\capital{\betstrat, c}{\bullet}\in\oscapprocs\) that superhedges~\(f\) on~\(\tostop{\stopt}{\pth}\), there is some bounded below capital process~\(\capital{\betstrat', c'}{\bullet}\in\oscapprocs\) with \(c'=\capital{\betstrat, c}{\stopt}\gr{\pth}\) that superhedges~\(f\gr{\pth\oplus_\stopt\bullet}\).
From this, we infer that indeed
\begin{equation*}
\oscuprev{f}\gr{\pth}
\geq \osuprev{f\gr{\pth\oplus_\stopt \bullet}}.
\end{equation*}
\fi
\qed
\end{proof}

The Markov property---see, for example, \cite[Ch.~4, Eq.~(1.2)]{1986EthierKurtz-Markov}---follows almost immediately from this memoryless character of the conditional upper expectation: for all time points~\(t\in\posreals\) and time periods~\(\Delta\in\posreals\),
\begin{equation*}
\oscuprev{g\gr{N_{t+\Delta}}}[t]\gr{\pth}
= \osuprev{g\gr{\pth\gr{t}+N_\Delta}}
\quad\text{for all } g\in\extreals^{\posints}, \pth\in\countpaths;
\end{equation*}
so does the strong Markov property \cite[Ch.~4, Eq.~(1.17)]{1986EthierKurtz-Markov}: for all stopping times~\(\stopt\in\stopts\) and time periods~\(\Delta\in\posreals\),
\begin{equation*}
\oscuprev{g\gr{N_{\stopt+\Delta}}}\gr{\pth}
= \osuprev{g\gr{\pth\gr{\stopt}+N_\Delta}}
\quad\text{for all } g\in\extreals^{\posints}, \pth\in\set{\stopt<+\infty}.
\end{equation*}


\subsection{The Law of Iterated Upper Expectations}
Next up is the crucial \emph{law of iterated (upper) expectations}, also known as the tower property.
For the sake of simplicity, we'll only establish a version of this law for constant stopping times and bounded and so-called finitary variables: those that depend only on the value of the counting path at finitely many time points.
Obviously, a variable~\(f\in\extvars\) is bounded and finitary if and only if \(f=g\gr{N_{t_1}, \dots, N_{t_k}}\) for some \(k\in\nats\), \(t_1\leq\cdots\leq t_k\in\posreals\) and some \(g\in\bfns[k]\), where here and in the remainder we let \(\bfns[k]\) be the set of bounded real functions on~\(\gr{\posints}^k\)---for \(k=1\), we'll simply write \(\bfns\).
\begin{theorem}\label{the:law of iterated upper expectation}
For all natural numbers~\(k\in\nats\), time points~\(t_1, \dots, t_{k+1}\in\posreals\) such that \(t_1<t_2<\cdots<t_{k+1}\) and gambles~\(g\in\bfns[k+1]\),
\begin{equation*}
\oscuprevstar{g\gr{N_{t_1}, \dots, N_{t_{k+1}}}}[t_1]
= \oscuprevstar{ \oscuprev{g\gr{N_{t_1}, \dots, N_{t_{k+1}}}}[t_k] }[t_1].
\end{equation*}
\end{theorem}
My proof is inspired by Shafer \& Vovk's~\cite[Proposition~8.7]{2001ShaferVovk-Probability} proof for a similar result in discrete time, but adds a move to ensure the `cut' is finite.
\ifipmu
\else
Because it's rather lengthy, I've relegated it to \zcref{asec:relegated-proofs}.
\fi

One can turn the law of iterated upper expectations in~\zcref{the:law of iterated upper expectation} into a (theoretical) recursive computational scheme to compute~\(\oscuprev{g\gr{N_{t_1}, \dots, N_{t_k}}}[t_1]\), at least once one realises that due to \zcref{prop:shifted upper expectation}, the variable
\[\vpth\mapsto\oscuprev{g\gr{N_{t_1}, \dots, N_{t_{k+1}}}}[t_k]\gr{\vpth}\]
in the conditional upper expectation on the right-hand side of the equality in \zcref{the:law of iterated upper expectation} is only functionally dependent on the values~\(\vpth\) takes in~\(t_1\), \dots, \(t_k\) rather than on the values it takes on the entire interval~\(\cci{0}{t_k}\)---so only on a countable set rather than an uncountable one.
\begin{corollary}\label{cor:trick with values in conditional}
For all natural numbers~\(k\in\nats\), time points~\(t_1, \dots, t_{k+1}\in\posreals\) such that \(t_1<\cdots<t_{k+1}\), bounded functions~\(g\in\bfns[k+1]\) and paths~\(\pth\in\countpaths\), and with \(\Delta\coloneqq t_{k+1}-t_k\),
\begin{equation*}
\oscuprev{g\gr{N_{t_1}, \dots, N_{t_{k+1}}}}[t_k]\gr{\pth}
= \osuprev{g\gr{\pth\gr{t_1}, \dots, \pth\gr{t_k}, \pth\gr{t_k}+N_{\Delta}}}.
\end{equation*}
\end{corollary}

\subsection{Connection to the Sublinear Poisson Semigroup}
\zcref[S]{cor:trick with values in conditional} tells us that it's important to be able to compute conditional upper expectations of the form~\(\oscuprev{g\gr{N_t}}[s]\) for bounded functions~\(g\in\bfns\) and time points~\(s,t\in\posreals\) such that \(s<t\).
For the classical Poisson process with rate~\(\lambda\), and for Vovk's measure-theoretic Poisson process \cite[Theorem~5]{1993Vovk}, these conditional expectations are related to the Poisson distribution~\(\psi_{\lambda\gr{t-s}}\) with parameter~\(\lambda\gr{t-s}\):
\begin{equation*}
\tscuprev{g\gr{N_t}}[s]\gr{\pth}
= \tsclprev{g\gr{N_t}}[s]\gr{\pth}
= \sum_{z\in\posints} g\gr{\pth\gr{s}+z} \psi_{\lambda\gr{t-s}}\gr{z}.
\end{equation*}
To generalise this property to our imprecise setting, it will be elucidative to express the right-hand side of this equality using the so-called \emph{Poisson semigroup}~\(\gr{\poisssg_\Delta}_{\Delta\in\posreals}\) for~\(\Delta=t-s\):
\begin{equation*}
\tscuprev{g\gr{N_t}}[s]\gr{\pth}
= \tsclprev{g\gr{N_t}}[s]\gr{\pth}
= \br{\poisssg_{\gr{t-s}} g}\gr{\pth\gr{s}}.
\end{equation*}
This family~\(\poisssg_\bullet\) of linear operators---linear maps from \(\bfns\) to \(\bfns\)---is generated by the Poisson generator~\(\poissgen\colon\bfns\to\bfns\), which maps any \(g\in\bfns\) to
\begin{equation*}
\poissgen g
\colon \posints \to \reals
\colon n\mapsto \lambda\gr{g\gr{n+1}-g\gr{n}},
\end{equation*}
as follows:
\begin{equation}\label{eq:poisssg as exponential}
\poisssg_\Delta
= e^{\Delta\poissgen}
= \lim_{k\to+\infty} \gr*{\eye+\frac{\Delta}{k}\poissgen}^k
\quad\text{for all } \Delta\in\posreals.
\end{equation}

Crucially, a similar result holds for the imprecise game-theoretic Poisson process, at least if we replace the Poisson semigroup~\(\poisssg_\bullet\) by the sublinear Poisson semigroup~\(\upoisssg_\bullet\).
As explained in \cite{2025Erreygers}, this family of sublinear operators is generated by the sublinear Poisson generator~\(\upoissgen\colon\bfns\to\bfns\), which maps any \(g\in\bfns\) to
\begin{equation*}
\upoissgen g
\colon \posints\to\reals
\colon n\mapsto \max\set{\lambda\gr[\big]{g\gr{n+1}-g\gr{n}}\colon \lambda\in\set{\llambda,\ulambda}},
\end{equation*}
through the operator exponential:
\begin{equation}\label{eq:upoisssg as exponential}
    \upoisssg_\Delta
    \coloneqq e^{\Delta\upoissgen}
    = \lim_{k\to+\infty} \gr*{\eye+\frac{\Delta}{k}\upoissgen}^k
    \quad\text{for all } \Delta\in\posreals.
\end{equation}
\begin{theorem}\label{the:with sublinear poisson semigroup}
For all time points~\(s,t\in\posreals\) such that \(s\leq t\) and \(g\in\bfns\),
\begin{equation*}
\oscuprev{g\gr{N_t}}[s]
= \br{\upoisssg_{t-s} g}\gr{N_s}.
\end{equation*}
\end{theorem}
As an intermediary step towards proving this result, we establish the following `generalisation'\footnote{We only generalise their result in the particular case \(A_t=\lambda t\).} of Vovk's~\cite{1993Vovk} Theorem~5; the proof is a rather straightforward modification of the original one, essentially replacing the `Bernouilli approximation' of~\zcref{eq:poisssg as exponential} with that of~\zcref{eq:upoisssg as exponential} and replacing the derived two-sided trading strategy with the natural one-sided counterpart.
\begin{proposition}\label{prop:sublinear not greater than watanabe}
For all \(s,t\in\posreals\) such that \(s\leq t\) and \(g\in\bfns\),
\begin{equation*}
\oscuprev{g\gr{N_t}}[s]
\leq \br[\big]{\upoisssg_{t-s}g}\gr{N_s}.
\end{equation*}
\end{proposition}
Proving \zcref{the:with sublinear poisson semigroup} is now a matter of combining \zcref{prop:sublinear not greater than watanabe} with a couple of other previously obtained results.
\begin{proof}[Proof of \zcref{the:with sublinear poisson semigroup}]
Recall from \zcref{prop:shifted upper expectation} that
\begin{equation*}
\oscuprev{g\gr{N_t}}[s]\gr{\pth}
= \osuprev{g\gr{\pth\gr{s}+N_{t-s}}}
\quad\text{for all } \pth\in\countpaths.
\end{equation*}
Consequently, it suffices to show that
\begin{equation*}
\uprev\br{g\gr{z+N_\Delta}}
= \br{\upoisssg_\Delta g}\gr{z}
\quad\text{for all } g\in\bfns, z\in\posints, \Delta\in\posreals.
\end{equation*}
So for all \(\Delta\in\posreals\), let \(\utranop_\Delta\colon\bfns\to\bfns\) map \(g\in\bfns\) to
\begin{equation*}
\utranop_\Delta g
\colon \posints\to\reals
\colon z\mapsto \uprev\br{g\gr{z+N_\Delta}}.
\end{equation*}
It follows immediately from the properties of the upper expectation that \(\utranop_\Delta\) is a sublinear transition operator, from \zcref{prop:sublinear not greater than watanabe} and \cite[Lemma~53]{2019Erreygers} that \(\utranop_\Delta\) is dominated by~\(\upoisssg_\Delta\), and from \ref{gencuprev:better bounds} that \(\utranop_0=\eye\).

Furthermore, we can use the law of iterated upper expectations [\zcref{the:law of iterated upper expectation}] to show that \(\utranop_\bullet\) is a semigroup.
To this end, fix some \(\Delta_1, \Delta_2\in\sposreals\), \(g\in\bfns\) and \(z\in\posints\), and observe that
\begin{equation*}
\br{\utranop_{\Delta_1+\Delta_2} g}\gr{z}
= \osuprev{g\gr{z+N_{\Delta_1+\Delta_2}}}
= \osuprev{\oscuprev{g\gr{z+N_{\Delta_1+\Delta_2}}}[\Delta_1]}.
\end{equation*}
Now for all \(\pth\in\countpaths\), it follows from \zcref{prop:shifted upper expectation} that
\begin{equation*}
\oscuprev{g\gr{z+N_{\Delta_1+\Delta_2}}}[\Delta_1]\gr{\pth}
= \osuprev{g\gr{z+\pth\gr{\Delta_1}+N_{\Delta_2}}}
= \br{\utranop_{\Delta_2} g}\gr{z+\pth\gr{\Delta_1}}.
\end{equation*}
Substituting into the preceding equality, we find that
\begin{align*}
\br{\utranop_{\Delta_1+\Delta_2} g}\gr{z}
= \uprev\br[\big]{\br{\utranop_{\Delta_2} g}\gr{z+N_{\Delta_1}}}
= \br{\utranop_{\Delta_1} \utranop_{\Delta_2} g}\gr{z},
\end{align*}
as required.

For any \(\lambda\in\cci{\llambda}{\ulambda}\) and \(\Delta\in\posreals\), let \(\tranop^\lambda_\Delta\) map~\(g\in\bfns\) to
\begin{equation*}
\tranop^\lambda_\Delta g
\colon \posints\to\reals
\colon z\mapsto \tsuprev{g\gr{z+N_\Delta}}.
\end{equation*}
Now from \zcref{prop:sublinear not greater than watanabe}, it follows that for all \(\Delta\in\posreals\) and \(g\in\bfns\),
\begin{equation*}
\tranop^\lambda_\Delta g
\leq \poisssg^\lambda_\Delta g
\quad\text{and}\quad
-\tranop^\lambda_\Delta \gr{-g}
\geq -\poisssg^\lambda_\Delta \gr{-g}
= \poisssg^\lambda_\Delta g,
\end{equation*}
where \(\poisssg^\lambda_\bullet\) is the Poisson semigroup with rate~\(\lambda\).
Consequently, \(\tranop^\lambda_\Delta=\poisssg^\lambda_\Delta\) for all \(\Delta\in\posreals\) and \(\lambda\in\cci{\llambda}{\ulambda}\).
From this, \zcref{prop:monotonicity in lambda,prop:sublinear not greater than watanabe}, it follows that
\begin{equation*}
\poisssg^\lambda_\Delta g
= \tranop^\lambda_\Delta g
\leq \utranop_\Delta g
\leq \upoisssg_\Delta g
\quad\text{for all } \lambda\in\cci{\llambda}{\ulambda}, \Delta\in\posreals, g\in\bfns.
\end{equation*}
Since \(\upoisssg_\bullet\) is equal to Nendel's~\cite{2021Nendel} so called \emph{Nisio semigroup} induced by the family \(\set{\poisssg^\lambda_\bullet\colon \lambda\in\cci{\llambda}{\ulambda}}\) \cite[Proposition~5.2]{2025Erreygers}, and this Nisio semigroup is the point-wise smallest semigroup that dominates the family \(\set{\poisssg^\lambda_\bullet\colon \lambda\in\cci{\llambda}{\ulambda}}\) \cite[Remark~5.3]{2021Nendel}, it follows from these inequalities that \(\utranop_\Delta=\upoisssg_\Delta\) for all \(\Delta\in\posreals\), which concludes our proof.
\qed
\end{proof}

\section{Conclusion}\label{sec:conclusion}

I am by no means the first to construct an `imprecise' Poisson process.
To the best of my knowledge, Hu and Peng~\cite[Section~6]{2021Hu} were the first to construct a sublinear expectation for \(\reals^d\)-valued Lévy processes, through the viscosity solution to an integro-partial differential equation characterised by a set of Lévy triples; see their Example~43 for the special case of the Poisson process.
This inspired Neufeld and Nutz \cite{2017Neufeld} to construct an upper expectation for \(\reals^d\)-valued Lévy processes through the upper envelope of the expectations with respect to the set of probability measures whose `differential characteristics' take values in the set of triplets; their Example~2.6 treats the Poisson process as a special case.

Unaware of this work, and inspired by work on imprecise finite-state Markov processes \cite{2017Krak}, Jasper De Bock and I~\cite{2019Erreygers} constructed a sublinear expectation for the Poisson process as the upper envelope of the expectations with respect to the set of `compatible' counting processes---a much less technical condition than the one in \cite{2017Neufeld}.
Inspired by \cite{2018Denk}, I then took an arguably more fruitful approach to the problem in~\cite[Section~6]{2024Erreygers}, constructing a sublinear expectation directly from the sublinear Poisson semigroup.

These earlier approaches have in common with the approach in the present contribution that the upper expectation of~\(g\gr{N_t}\) can be retrieved as the (viscosity) solution to what is essentially the same differential equation; compare the differential equations in~\cite[Eq.~(9)]{2021Hu} and \cite[Eq.~(2.5)]{2017Neufeld} to the one in~\cite[Proposition~3.6]{2025Erreygers}.
Except for~\cite{2024Erreygers}, these earlier approaches suffer from the same limitation: they only (practically) determine the upper expectation of finitary variables.

Under the game-theoretic approach, the upper expectation of such (bounded) finitary variables can be determined by means of \zcref{the:law of iterated upper expectation,cor:trick with values in conditional,the:with sublinear poisson semigroup}.
However, in \zcref{prop:expected renewal time} I've also explicitly determined the upper and lower expectation of a non-finitary variable---the time~\(\rho_\stopt-\stopt\) until the next increment.
Another benefit of the game-theoretic approach over the other approaches is that it involves fewer technicalities---there's no need for measurability, for example, and conditional expectations are defined uniquely and naturally.

Whether the (conditional) upper expectation of other non-finitary variables can also be determined easily is one of the many possible lines of follow-up research I envision.
A related future research topic is the continuity of the game-theoretic upper expectation with respect to point-wise convergence of variables.
More generally, one could see this work as a first step in the study of imprecise renewal processes, or even imprecise countable-state Markov processes, in the game-theoretic framework.

\begin{credits}
\subsubsection{\ackname}
My work on this contribution was funded by Ghent University's Special Research Fund.
I'd like to express my gratitude to the three anonymous reviewers; their constructive comments have significantly improved this contribution.
I dedicate this contribution to Milena and Dolores, in part for their putting up with my mental and/or physical absence during the preparation of this contribution, but mostly for their love.

\subsubsection{\discintname} Nothing to disclose.
\end{credits}

\bibliographystyle{splncs04}
\nocite{2007Williams} 
\bibliography{game-theoretic-poisson}
\ifipmu\else
\newpage
\appendix
\section{Relegated proofs}\label{asec:relegated-proofs}

\subsection{Additional results and relegated proofs for results in \zcref{sec:game-theoretic upper expectations}}
\begin{relproof}{\zcref{prop:gencuprev}}
Note that since \(\gencapprocs\) contains all constant processes and is a cone, the same is true for the set
\begin{equation*}
\set{\gencapital_\bullet\in\gencapprocs\colon \inf\gencapital_\bullet>-\infty},
\end{equation*}
which is non-empty because it contains at least the constant processes.
The properties follow almost immediately from this observation and the definition of~\(\gencuprev{\bullet}[\bullet]\).
\end{relproof}

One consequence of coherence is that in the definition of the conditional upper expectation~\(\gencuprev{\bullet}\gr{\pth}\), we can limit ourselves to capital processes that are always greater than the infimum of~\(f\) on~\(\tostop{\stopt}{\pth}\).
\begin{lemma}\label{lem:capital above remaining infimum}
Suppose \(\gencapprocs\) is coherent.
Then for all \(\stopt\in\stopts\), \(f\in\stoptdom\), \(\pth\in\set{\stopt<+\infty}\) and \(\gencapital_\bullet\in\gencapprocs\) such that \(\liminf\gencapital_\bullet\geq_{\tostop{\stopt}{\pth}} f\),
\begin{equation*}
\gencapital_t\vert_{\tostop{\stopt}{\pth}}
\geq \inf f\vert_{\tostop{\stopt}{\pth}}
\quad\text{for all } t\in\coi{\stopt\gr{\pth}}{+\infty}.
\end{equation*}
\end{lemma}
\begin{proof}
Follows almost immediately from the definition of coherence.
\qed\end{proof}

\zcref[S]{lem:capital above remaining infimum} comes in handy in our proof for \zcref{prop:gencuprev with coherence}, more specifically when proving \ref{gencuprev:better bounds}.
\begin{relproof}{\zcref{prop:gencuprev with coherence}}
For \ref{gencuprev:better bounds}, the upper found follows almost immediately from the definition of \(\gencuprev{\bullet}\gr{\pth}\) and the assumption in the statement that \(\gencapprocs\) contains all constants, while the lower bound follows from \zcref{lem:capital above remaining infimum} for \(t=\stopt\gr{\pth}\).

For \ref{gencuprev:lower upper}, we need to prove that
\begin{equation}\label{eqn:proof of gencuprev with coherence:to prove}
\genclprev{f}\gr{\pth}
= -\gencuprev{-f}\gr{\pth}
\leq \gencuprev{f}\gr{\pth}.
\end{equation}
This inequality holds trivially if \(\gencuprev{f}\gr{\pth}+\infty\) or \(\gencuprev{-f}\gr{\pth}=+\infty\), so we assume that \(\gencuprev{f}\gr{\pth}<+\infty\) and \(\gencuprev{-f}\gr{\pth}<+\infty\).
Since \(f-f\geq 0\), it follows from \ref{gencuprev:subadditivity}, \ref{gencuprev:monotonicity} and \ref{gencuprev:better bounds} that
\begin{align*}
\gencuprev{f}\gr{\pth}+\gencuprev{-f}\gr{\pth}
\geq \gencuprev{f-f}\gr{\pth}
\geq \gencuprev{0}\gr{\pth}
= 0.
\end{align*}
Since \(\gencuprev{f}\gr{\pth}<+\infty\) and \(\gencuprev{-f}\gr{\pth}<+\infty\), it follows from this inequality that \(\gencuprev{f}\gr{\pth}>-\infty\) and \(\gencuprev{-f}\gr{\pth}>-\infty\); that is, \(\gencuprev{f}\gr{\pth}\) and \(\gencuprev{-f}\gr{\pth}\) are both real.
Consequently, the inequality above indeed implies the one in~\eqref{eqn:proof of gencuprev with coherence:to prove}.
\end{relproof}

\subsection{Proofs for results in \zcref{sec:betting strategies}}

\begin{relproof}{\zcref{prop:oscaprocss:co co co}}
That \(\oscapprocs\) contains the constants follows from choosing \(\stopt_1=0=\stopt_2\) in the definition of the trading strategy~\(\betstrat\), and that it's closed under multiplication with positive scalars clearly follows immediately from the definition of trading strategies.
That \(\oscapprocs\) is closed under pointwise addition takes a bit more work to prove formally, but should be intuitively clear.
We'll only prove explicitly that \(\oscapprocs\) is coherent.
To this end, we fix some \(\capital{\betstrat, c}{\bullet}\in\oscapprocs\), \(t\in\posreals\) and \(\pth\in\countpaths\), and set out to show that for any \(\epsilon\in\sposreals\), there is some \(\vpth\in\tostop{t}{\pth}\) such that
\begin{equation}\label{eqn:proof of no sure gain again:inequality with epsilon}
\liminf\capital{\betstrat, c}{\bullet}\gr{\vpth}
< \capital{\betstrat, c}{t}\gr{\pth} + \epsilon.
\end{equation}
So fix any such \(\epsilon\in\sposreals\).

We'll show the existence of the path~\(\vpth\) by constructing it recursively such that the capital doesn't increase by more than~\(\epsilon\).
To this end, we let \(\mathscr{K}\) be the set of indices~\(k\in\set{1, \dots, n}\) such that \(t<\stopt_{k+1}\gr{\pth}\).
In case this index set~\(\mathscr{K}\) is empty, \(\capital{\betstrat, c}{r}\gr{\pth}=\capital{\betstrat, c}{t}\gr{\pth}\) for all \(r\in\coi{t}{+\infty}\), so the path~\(\vpth=\pth\) satisfies the inequality~\eqref{eqn:proof of no sure gain again:inequality with epsilon}.
If \(\mathscr{K}\neq\emptyset\), we let \(k\coloneqq \min\mathscr{K}\) and distinguish two subcases.

If \(\stopt_k\gr{\pth}>t\), let \(\vpth_k\coloneqq\pth\).
Otherwise, we set out to construct some counting path~\(\vpth_k\in\tostop{t}{\pth}\) for which the `currently active' bets with stakes~\(\ustake{k}\gr{\pth}\) and \(\lstake{k}\gr{\pth}\) don't increase the capital after~\(t\) by more than \(\epsilon/n\).
Choose some \(\delta_k\in\sposreals\) such that \(\pth\gr{t+\delta_k}=\pth\gr{t}\) [this is possible because \(\pth\) is continuous from the right], \(-\diffstake{k}\gr{\pth}\llambda\delta_k<\epsilon/n\), \(\delta_k<\stopt_{k+1}\gr{\pth}-t\) and  \(\llambda\delta_k\leq 1\).
It follows from \zcref{eqn:one-sided capital process different increment} that the counting path~\(\vpth_k\) given by
\begin{equation*}
r\mapsto \begin{cases}
\pth\gr{r} & \text{if } r<t+\delta_k,\\
\pth\gr{t}& \text{if } r\geq t+\delta_k \text{ and } \diffstake{k-1}\gr{\pth}\geq 0\\
\pth\gr{t} + \lceil\llambda\gr{r-t}\rceil & \text{if } r\geq t+\delta_k \text{ and } \diffstake{k-1}\gr{\pth}<0
\end{cases}
\end{equation*}
does exactly that.
Note that \(\stopt_{k+1}\gr{\vpth_k}\geq t+\delta_k\)---if this were not the case, then since \(\vpth_k\vert_{\cci{0}{\stopt_{k+1}\gr{\vpth_k}}}=\pth\vert_{\cci{0}{\stopt_{k+1}\gr{\vpth_k}}}\) by definition and \(\stopt_{k+1}\) is a stopping time, it must be that \(\stopt_{k+1}\gr{\pth}=\stopt_{k+1}\gr{\vpth_k}<t+\delta_k\), which contradicts our requirement that \(\delta_k<\stopt_{k+1}\gr{\pth}-t\).
Similarly, note that our construction also ensures that \(\ustake{k}\gr{\vpth_k}=\ustake{k}\gr{\pth}\) and \(\lstake{k}\gr{\vpth_k}=\lstake{k}\gr{\pth}\), as the stakes~\(\ustake{k}\) and \(\lstake{k}\) are \(\stopt_k\)-measurable and \(\vpth_k\vert_{\cci{0}{\stopt_{k}\gr{\pth}}}=\pth\vert_{\cci{0}{\stopt_{k}\gr{\pth}}}\).

If \(\stopt_{k+1}\gr{\vpth_k}=+\infty\) or \(k=n\), our construction ensures that \(\capital{\betstrat, c}{r}\gr{\vpth_k}<\capital{\betstrat, c}{t}\gr{\pth}+\epsilon/n\) for all \(r\in\coi{t}{+\infty}\), which implies the inequality in~\eqref{eqn:proof of no sure gain again:inequality with epsilon}.
If on the other hand \(k<n\) and \(\stopt_{k+1}\gr{\vpth_k}<+\infty\), we let \(\ell\) be the smallest index in~\(\set{k+1, \dots, n}\) such that \(\stopt_\ell\gr{\vpth_k}<\stopt_{\ell+1}\gr{\vpth_k}\).
We repeat the same argument as before, but modifying \(\vpth_k\) from \(\stopt_\ell\gr{\vpth_k}\) onwards rather than \(\pth\) from~\(t\) onwards.
After at most \(n-k\) repetitions of the same argument, we'll end up with a counting path in~\(\tostop{t}{\pth}\) that satisfies the inequality in~\eqref{eqn:proof of no sure gain again:inequality with epsilon}.
\end{relproof}

\subsection{Additional results and relegated proofs for results in \zcref{sec:properties}}
\begin{relproof}{\zcref{prop:monotonicity in lambda}}
It suffices to make the following observation.
Fix any bounded below capital process~\(\capital{\betstrat_\mathrm{o}, c_{\mathrm{o}}}{\bullet}\in\oscapprocs[\cci{\llambdao}{\ulambdao}]\) and any path~\(\pth\in\set{\stopt<+\infty}\).
Then with \(c_{\mathrm{i}}\coloneqq\capital{\betstrat_\mathrm{o}, c_{\mathrm{o}}}{\stopt}\gr{\pth}\) and \(\betstrat_\mathrm{i}\) a copy of the trading strategy~\(\betstrat_\mathrm{o}\) but with stopping times~\(\stopt_k^\mathrm{i}\coloneqq\stopt_k^\mathrm{o}\vee\stopt\) instead of~\(\stopt_k^\mathrm{o}\), the capital process~\(\capital{\betstrat_{\mathrm{i}}, c_{\mathrm{i}}}{\bullet}\in\oscapprocs[\cci{\llambdai}{\ulambdai}]\) has capital~\(c_{\mathrm{i}}\) in~\(\stopt\) for~\(\pth\) and is constructed in such a way that for all \(t\in\ooi{\stopt\gr{\pth}}{+\infty}\) and \(\vpth\in\tostop{\stopt}{\pth}\),
\begin{multline*}
\capital{\betstrat_{\mathrm{i}}, c_{\mathrm{i}}}{t}\gr{\vpth}-\capital{\betstrat_{\mathrm{o}}, c_\mathrm{o}}{t}\gr{\vpth} \\
= \sum_{k=1}^n \gr{\stopt_{k+1}^\mathrm{i}\gr{\vpth}\wedge t-\stopt_k^\mathrm{i}\gr{\vpth}\wedge t} \gr*{\ustake{k}\gr{\vpth}\gr{\ulambdao-\ulambdai}+\lstake{k}\gr{\vpth}\gr{\llambdai-\llambdao}}
\geq 0,
\end{multline*}
where the inequality follows from the assumptions in the statement.
So \(\capital{\betstrat_{\mathrm{i}}, c_{\mathrm{i}}}{\bullet}\) is uniformly bounded below and dominates~\(\capital{\betstrat_\mathrm{o}, c_{\mathrm{o}}}{\bullet}\) on~\(\tostop{\stopt}{\pth}\).
\end{relproof}

The following corollary of \zcref{prop:expected increment} will be useful in the proof of \zcref{the:law of iterated upper expectation} and \zcref{prop:sublinear not greater than watanabe}---through \zcref{lem:more than one jump in small interval}---further on.
\begin{corollary}\label{cor:bound on number of jumps}
For any two time points~\(s,t\in\posreals\) such that \(s\leq t\) and \(m\in\nats\),
\begin{equation*}
\oscuprev{\indica{\set{N_t-N_s\geq m}}}[s]
\leq \frac{\ulambda\gr{t-s}}m.
\end{equation*}
\end{corollary}
\begin{proof}
Since \(\indica{\set{N_t-N_s\geq m}}\leq\gr{N_t-N_s}/m\), this follows immediately from \zcref{prop:expected increment} and \ref{gencuprev:monotonicity} \& \ref{gencuprev:positive homogeneity} in \zcref{prop:gencuprev}.
\qed\end{proof}

We defer proving \zcref{prop:expected renewal time} to \zcref{asec:proof of expected renewal time} further on.
The next item on our menu of unproven results is \zcref{the:law of iterated upper expectation}.
Our proof for this result will simplify quite a bit thanks to the following intermediary result and \zcref{cor:trick with values in conditional} to \zcref{prop:shifted upper expectation}.
\begin{lemma}\label{lem:stopt-measurable then stop there}
Fix some \(s,t\in\posreals\) such that \(s\leq t\).
Then for all \(t\)-measurable~\(f\in\extvars\) and \(\pth\in\countpaths\),
\begin{multline*}
\oscuprev{f}[s]\gr{\pth}
= \inf\Big\{\capital{\betstrat, c}{s}\gr{\pth}\colon \capital{\betstrat, c}{\bullet}\in\oscapprocs, \inf\capital{\betstrat, c}{\bullet}>-\infty, \\ \gr{\forall \vpth\in\tostop{s}{\pth}}~\capital{\betstrat, c}{t}\gr{\vpth}\geq f\gr{\vpth}\Big\}.
\end{multline*}
\end{lemma}
\begin{proof}
Since \(\oscapprocs\) is coherent, it follows immediately from \zcref{lem:capital above remaining infimum} and the \(t\)-measurability of~\(f\) that for any bounded below capital process \(\capital{\betstrat, c}{\bullet}\in\oscapprocs\) that superhedges~\(f\) on~\(\tostop{s}{\pth}\),
\begin{equation*}
\capital{\betstrat, c}{t}\gr{\vpth}
\geq \inf f\vert_{\tostop{t}{\vpth}}
= f\gr{\vpth}
\quad\text{for all }
\vpth\in\tostop{s}{\pth}.
\end{equation*}

Conversely, suppose there is some bounded one-sided capital process~\(\capital{\betstrat, c}{\bullet}\in\oscapprocs\) such that \(\capital{\betstrat, c}{t}\gr{\vpth}\geq f\gr{\vpth}\) for all \(\vpth\in\tostop{s}{\pth}\).
Let \(\betstrat'\) be the betting strategy that copies~\(\betstrat\) but replaces each stopping time~\(\stopt_k\) by~\(\stopt_k'\coloneqq\stopt_k\wedge t\).
Then \(\betstrat'\) is a valid trading strategy, and \(\capital{\betstrat', c}{\bullet}\in\oscapprocs\) is bounded below and defined in such a way that
\begin{equation*}
\capital{\betstrat', c}{t}\gr{\vpth}
= \capital{\betstrat, c}{t}\gr{\vpth}
\quad\text{for all } \vpth\in\countpaths.
\end{equation*}
Consequently,
\begin{equation*}
\liminf\capital{\betstrat', c}{\bullet}\gr{\vpth}
= \capital{\betstrat', c}{t}\gr{\vpth}
= \capital{\betstrat', c}{t}\gr{\vpth}
\geq f\gr{\vpth}
\quad\text{for all } \vpth\in\tostop{s}{\pth}.
\end{equation*}
This concludes our proof for the equality in the statement.
\qed\end{proof}
\begin{relproof}{\zcref{cor:trick with values in conditional}}
This follows immediately from \zcref{prop:shifted upper expectation} once one realises that for all \(\vpth\in\countpaths\),
\begin{equation*}
\br[\big]{g\gr{N_{t_1}, \dots, N_{t_{k+1}}}}\gr{\pth\oplus_{t_k}\vpth}
= g\gr{\pth\gr{t_1}, \dots, \pth\gr{t_k}, \pth\gr{t_k}+\vpth\gr{t_{k+1}-t_k}}.
\end{equation*}
\end{relproof}
\begin{relproof}{\zcref{the:law of iterated upper expectation}}
Fix any \(\pth\in\countpaths\) and \(\epsilon\in\sposreals\), and let \(f\coloneqq g\gr{N_{t_1}, \dots, N_{t_{k+1}}}\).
Since \(f=\gr{f-\inf f}+\inf f\) and \(\inf\gr{f-\inf f}\geq 0\), thanks to \ref{gencuprev:constant additivity} we may assume without loss of generality that \(\inf f\geq 0\).
Furthermore, since \(f\) is positive and bounded, it follows from \ref{gencuprev:better bounds} that so are \(\oscuprev{f}[t_1]\) and \(\oscuprev{f}[t_k]\).

Since \(f\) is \(t_{k+1}\)-measurable, it follows from \zcref{lem:stopt-measurable then stop there} that there is some capital process~\(\capital{\betstrat_1, c_1}{\bullet}\in\oscapprocs\) that is bounded below with \(\oscuprev{f}[t_1]\gr{\pth}\leq\capital{\betstrat_1, c_1}{t_1}\gr{\pth}<\oscuprev{f}[t_1]\gr{\pth}+\epsilon\) and \(\capital{\betstrat_1, c_1}{t_{k+1}}\geq_{\tostop{t_1}{\pth}} f\).
Consequently, it must be that \(\oscuprev{f}[t_k]\gr{\vpth}\leq\capital{\betstrat_1, c_1}{t_k}\gr{\vpth}\) for all \(\vpth\in\tostop{t_k}{\pth}\); thanks to \ref{gencuprev:monotonicity} and \zcref{lem:stopt-measurable then stop there}, we infer from this that
\begin{equation*}
\oscuprevstar{\oscuprev{f}[t_k]}[t_1]\gr{\pth}
\leq \oscuprevstar{\capital{\betstrat_1, c_1}{t_k}}[t_1]\gr{\pth}.
\end{equation*}
Clearly, the \(t_k\)-measurable varibale~\(\capital{\betstrat_1, c_1}{t_k}\) is hedged on~\(\tostop{t_1}{\pth}\) (and at \(t_k\)) by \(\capital{\betstrat_1, c_1}{\bullet}\); consequently,
\begin{equation}\label{eqn:proof of law of iterated upper expectation:first inequality}
\oscuprevstar{\oscuprev{f}[t_k]}[t_1]\gr{\pth}
\leq \capital{\betstrat_1, c_1}{t_1}\gr{\pth}
< \oscuprev{f}[t_1]\gr{\pth}+\epsilon.
\end{equation}

For the converse inequality, recall from \zcref{sec:setup} that \(\oscuprev{f}[t_k]\) is \(t_k\)-measurable.
Hence, it follows from \zcref{lem:stopt-measurable then stop there} that there is some bounded below capital process~\(\capital{\betstrat, c}{\bullet}\in\oscapprocs\) such that
\begin{equation*}
\oscuprevstar{\oscuprev{f}[t_k]}[t_1]\gr{\pth}
\leq \capital{\betstrat, c}{t_1}\gr{\pth}
< \oscuprevstar{\oscuprev{f}[t_k]}[t_1]\gr{\pth}+\frac{\epsilon}3
\end{equation*}
and \(\capital{\betstrat, c}{t_{k}}\geq_{\tostop{t_1}{\pth}} \oscuprev{f}[t_k]\).
Let \(\Delta\coloneqq t_{k+1}-t_k\), and fix some natural number~\(m\) such that \(3\gr{\sup f}\ulambda\Delta<m\epsilon\).
Then by \zcref{cor:bound on number of jumps}, there is some positive capital process~\(\capital{\betstrat_0, c_0}{\bullet}\in\oscapprocs\) with \(\capital{\betstrat_0, c_0}{t_1}\gr{\pth}\leq\ulambda\Delta/m+\epsilon/\gr{3\sup f}\) such that \(\capital{\betstrat_0, c_0}{t_k}\geq1\) for all \(\vpth\in\tostop{t_1}{\pth}\) with \(\vpth\gr{t_k}\geq\vpth\gr{t_1}+m\).
On the other hand, for all \(z=z_{1:k}\in\gr{\posints}^k\) with \(\pth\gr{t_1}=z_1\leq z_2 \leq \cdots \leq z_k<\pth\gr{t_1}+m\), thanks to \zcref{lem:stopt-measurable then stop there} there is some positive capital process~\(\capital{\betstrat_z, c_z}{\bullet}\in\oscapprocs\) such that
\begin{equation*}
\osuprev{g\gr{z_1, \dots, z_k, z_k+N_{\Delta}}}
\leq
c_z
< \osuprev{g\gr{z_1, \dots, z_k, z_k+N_{\Delta}}} + \frac\epsilon3
\end{equation*}
and
\begin{equation*}
\capital{\betstrat_z, c_z}{\Delta}\gr{\vpth}
\geq g\gr{z_1, \dots, z_k, z_k+\vpth\gr{\Delta}}
\quad\text{for all } \vpth\in\countpaths.
\end{equation*}
Consider now the trading strategy~\(\betstrat'\) that trades according to~\(\betstrat\) until~\(t_k\) and, for \(\vpth\in\tostop{t_1}{\pth}\) with \(\vpth\gr{t_k}<\pth\gr{t_1}+m\), from \(t_k\) onwards trades according to~\(\betstrat_{\gr{\vpth\gr{t_1} \dots, \vpth\gr{t_k}}}\) shifted by~\(t_k\)---it shouldn't take too much effort from the reader to understand that this is still a valid trading strategy.
Let \(c'\coloneqq c+\epsilon/3\).

By construction, and thanks to \zcref{cor:trick with values in conditional}, \(\capital{\betstrat', c'}{t_1}\gr{\pth}=\capital{\betstrat, c}{t_1}\gr{\pth}+\epsilon/3\), and \(\capital{\betstrat', c'}{t_{k+1}}\gr{\vpth}\geq f\gr{\vpth}\) for all \(\vpth\in\tostop{t_1}{\pth}\) such that \(\vpth\gr{t_k}<\pth\gr{t_1}+m\).
Additionally, for all \(\vpth\in\tostop{t_1}{\pth}\)---so in particular for those with \(\vpth\gr{t_k}\geq\pth\gr{t_1}+m\)---\(\gr{\sup f}\capital{\betstrat_0, c_0}{t_k}\gr{\vpth}\geq\gr{\sup f}\indica{\set{N_{t_k}-N_{t_1}\geq m}}\gr{\vpth}\).
Since furthermore \(\oscapprocs\) is closed under positive linear combinations, it follows that
\begin{align}
\oscuprevstar{f}[t_1]\gr{\pth}
&\leq \capital{\betstrat', c'}{t_1}\gr{\pth} + \gr{\sup f}\capital{\betstrat_0, c_0}{t_1}\gr{\pth} \nonumber\\
&\leq \capital{\betstrat, c}{t_1}\gr{\pth}+\frac\epsilon3+\gr{\sup f}\frac{\ulambda\Delta}{m} + \frac\epsilon3 \nonumber\\
&< \capital{\betstrat, c}{t_1}\gr{\pth}+\epsilon \nonumber\\
&< \oscuprevstar{\oscuprev{f}[t_k]}[t_1]\gr{\pth}+\epsilon.
\label{eqn:proof of law of iterated upper expectation:second inequality}
\end{align}
Since \(\epsilon\) can be made arbitrarily small, the equality in the statement follows from the inequalities~\eqref{eqn:proof of law of iterated upper expectation:first inequality} and \eqref{eqn:proof of law of iterated upper expectation:second inequality}.
\end{relproof}

In our proof for \zcref{prop:sublinear not greater than watanabe}, we'll need a bound on the upper probability of having more than one jump in one of a sequence of consecutive intervals of the same length.
\begin{lemma}\label{lem:more than one jump in small interval}
Fix some \(s,t\in\posreals\) such that \(s<t\).
For all \(n\in\nats\), we define \(\Delta_n\coloneqq\gr{t-s}/n\), \(t^n_k\coloneqq s+\gr{k-1}\Delta_n\) for all \(k\in\set{1, \dots, n+1}\)  and \[A_n
\coloneqq \set{\pth\in\countpaths\colon \gr{\forall k\in\set{1, \dots, n}}~\pth\gr{t^n_{k+1}}\leq\pth\gr{t^n_k}+1 }.\]
Then
\begin{equation*}
\gr{\forall\epsilon\in\sposreals}
\gr{\exists n_\epsilon\in\nats}
\gr{\forall n\in\nats, n\geq n_\epsilon}~
\oscuprev{\indica{\compl{A_n}}}[s]
<\epsilon
\end{equation*}
\end{lemma}
\begin{proof}
Fix any \(\pth\in\countpaths\), \(\epsilon\in\sposreals\) and \(m\in\nats\) such that \(m>3\ulambda\gr{t-s}/\epsilon\).
Then by \zcref{cor:bound on number of jumps,lem:stopt-measurable then stop there}, there is some capital process~\(\capital{\betstrat_m, c_m}{\bullet}\in\oscapprocs\) such that \(\capital{\betstrat_m, c_m}{s}\gr{\pth}<\ulambda\gr{t-s}/m+\epsilon/3\) and \(\capital{\betstrat_m, c_m}{t}\gr{\vpth}\geq\indica{\set{N_t-N_s\geq m}}\gr{\vpth}\) for all \(\vpth\in\tostop{s}{\pth}\).

Now consider the elementary betting strategy~\(G\) that bets with unit stake from the moment there is some jump in the interval~\(\cci{t^n_k}{t^n_{k+1}}\) and stops at the end of this interval, and stops trading altogether once the path has increased by more than \(m\); more formally, we consider the stopping times \(\stopt_1\coloneqq\stopt'_1\wedge\sigma\), \dots, \(\stopt_{2m+1}\coloneqq\stopt'_{2m+1}\wedge\sigma\) with
\begin{equation*}
\sigma
\colon
\countpaths\to\posextreals
\colon \vpth\mapsto \inf\set{r\in\posreals\colon \vpth\gr{r}\geq \vpth\gr{s}+m}
\end{equation*}
and with \(\stopt'_1\), \dots, \(\stopt'_{2m+1}\) defined recursively by \(\stopt'_1\coloneqq t^n_1=s\) and, for all \(j\in\set{1, \dots, m}\), by
\begin{align*}
\stopt'_{2j}
&\colon \countpaths \to\posextreals
\colon \vpth\mapsto \inf \set*{r\in\posreals\colon \stopt_{2j-1}\gr{\vpth}<r\leq t, \vpth\gr{r}>\lim_{r_-\nearrow r} \vpth\gr{r_-}}
\shortintertext{and}
\stopt'_{2j+1}
&\colon \countpaths \to\posextreals
\colon \vpth\mapsto \inf \set[\big]{t^n_k\colon k\in\set{1, \dots, n+1}, t^n_k\geq \stopt_{2j}\gr{\vpth}},
\end{align*}
and the stakes \(\ustake{2j}\coloneqq 1\) , \(\ustake{2j-1}\coloneqq0\) and \(\lstake{2j}\coloneqq0\eqqcolon\lstake{2j-1}\).
Then with \(c\coloneqq\epsilon/3\), our construction ensures that \(\capital{\betstrat, c}{t}\gr{\vpth}\geq c-m\ulambda\Delta_n\) for all \(\vpth\in\countpaths\), and in particular that
\begin{equation*}
\capital{\betstrat, c}{t}\gr{\vpth}
\geq c - m\ulambda\Delta_n +1
\quad\text{for all } \vpth\in\set{N_t-N_s<m}\cap \compl{A_n}.
\end{equation*}
Let \(n_\epsilon\coloneqq 3m\ulambda\gr{t-s}/\epsilon\); then for all \(n\in\nats\) such that \(n\geq n_\epsilon\),
\begin{equation*}
m\ulambda\Delta_n
= m\ulambda\gr{t-s}/n
\leq \frac\epsilon3
= c.
\end{equation*}
Consequently, for all such~\(n\geq n_\epsilon\), \(\indica{\compl{A_n}}\) is superhedged (in~\(t\)) on~\(\tostop{s}{\pth}\) by \(\capital{\betstrat, c}{\bullet}+\capital{\betstrat_m, c_m}{\bullet}\in\oscapprocs\).
Due to \zcref{lem:stopt-measurable then stop there}, we infer from all this that
\begin{equation*}
\oscuprev{\indica{\compl{A_n}}}[s]\gr{\pth}
\leq \capital{\betstrat, c}{s}\gr{\pth}+\capital{\betstrat_m, c_m}{s}\gr{\pth}
< c+\frac{\ulambda\gr{t-s}}{m}+\frac\epsilon3
< \epsilon,
\end{equation*}
which is what we needed to prove.
\qed\end{proof}
\begin{relproof}{\zcref{prop:sublinear not greater than watanabe}}
In the degenerate case \(s=t\), the equality in the statement is immediate because (i) \(\upoisssg_0=\eye\); and (ii) \(\oscuprev{\bullet}[s]\) maps the \(s\)-measurable variable~\(g\gr{N_s}\) to itself due to \ref{gencuprev:better bounds}.
Henceforth, we therefore assume that \(s<t\).

We fix any \(\pth\in\countpaths\) and \(g\in\bfns\), and set out to prove that
\begin{equation*}
\oscuprev{g\gr{N_t}}[s]\gr{\pth}
\leq \br[\big]{\upoisssg_{t-s}g}\gr{\pth\gr{s}}.
\end{equation*}
For all \(n\in\nats\), let \(\Delta_n\), \(t^n_1\), \dots, \(t^n_{n+1}\) and \(A_n\) be defined as in  \zcref{lem:more than one jump in small interval}, and let \(g^n_k\coloneqq \gr{\eye+\Delta_n \upoissgen}^{n+1-k} g\) for all \(k\in\set{1, \dots, n+1}\).
Note that \(g^n_k=\gr{\eye+\Delta_n\upoissgen}g^n_{k+1}\) and that the stopping time
\begin{equation*}
\sigma_n
\colon \countpaths \to \posextreals
\colon \vpth\mapsto \inf\bigcup_{k=1}^n\set{r\in\cci{t^n_k}{t^n_{k+1}}\colon \vpth\gr{r}\geq \vpth\gr{t^n_k}+2}
\end{equation*}
is equal to \(+\infty\) on the event~\(A_n\).

Fix any \(\epsilon\in\sposreals\).
Then by \zcref{lem:more than one jump in small interval} and \cite[Theorem~3.1]{2025Erreygers}, there is some \(n\in\nats\) such that \(\delta\coloneqq\gr{\sup g-\inf g}\llambda\Delta_n<\epsilon\), \(0\leq\oscuprev{\indica{\compl{A_n}}}[s]\gr{\pth}<\epsilon/2(\sup g-\inf g)\) and \(\abs{\br{\upoisssg_{t-s}g}\gr{\pth\gr{s}}-g^n_1\gr{\pth\gr{s}}}<\epsilon\).
For this \(n\), we consider the initial capital~\(c\coloneqq g^n_1\gr{\pth\gr{s}}+\delta\) in combination with the elementary betting strategy~\(\betstrat\) with stopping times \(\stopt_1\coloneqq t^n_1\wedge\sigma_n\), \dots, \(\stopt_{n+1}\coloneqq t^n_{n+1}\wedge\sigma_n\) and stakes defined for all \(k\in\set{1, \dots, n}\) and \(\vpth\in\countpaths\) by
\begin{align*}
\ustake{k}\gr{\vpth}
&\coloneqq \gr[\Big]{+g^n_{k+1}\gr[\big]{\vpth\gr{\stopt_k}+1}-g^n_{k+1}\gr[\big]{\vpth\gr{\stopt_k}}}\vee 0
\shortintertext{and}
\lstake{k}\gr{\vpth}
&\coloneqq \gr[\Big]{-g^n_{k+1}\gr[\big]{\vpth\gr{\stopt_k}+1}+g^n_{k+1}\gr[\big]{\vpth\gr{\stopt_k}}}\vee 0.
\end{align*}
This way, for all \(\vpth\in\tostop{s}{\pth}\), \(k\in\set{1, \dots, n}\) with \(\stopt_k\gr{\vpth}<\sigma_n\gr{\vpth}\) and \(r\in\cci{\stopt_k\gr{\vpth}}{\stopt_{k+1}\gr{\vpth}}\) such that \(r<\sigma_n\gr{\vpth}\),
\begin{multline}
\label{eqn:proof of sublinear watanabe:capital change before sigma}
\capital{\betstrat, c}{r}\gr{\vpth}
- \capital{\betstrat , c}{\stopt_k}\gr{\vpth} \\
= g^n_{k+1}\gr{\vpth\gr{r}}-g^n_k\gr{\vpth\gr{\stopt_k}} + \gr[\big]{\ustake{k}\gr{\vpth}\ulambda-\lstake{k}\gr{\vpth}\llambda}\gr{t^n_{k+1}-r}.
\end{multline}
To verify this equality, observe that by construction of the capital process~\(\capital{\betstrat, c}{\bullet}\),
\begin{align*}
\MoveEqLeft
\capital{\betstrat, c}{r}\gr{\vpth}
- \capital{\betstrat , c}{\stopt_k}\gr{\vpth} \\
&= \ustake{k}\gr{\vpth}\gr[\big]{\vpth\gr{r}-\vpth\gr{\stopt_k}-\ulambda\gr{r-\stopt_k\gr{\vpth}}} \\&\qquad - \lstake{k}\gr{\vpth}\gr[\big]{\vpth\gr{r}-\vpth\gr{\stopt_k}-\llambda\gr{r-\stopt_k\gr{\vpth}}} \\
&= \ustake{k}\gr{\vpth}\gr[\big]{\vpth\gr{r}-\vpth\gr{\stopt_k}-\ulambda\Delta_n} + \ustake{k}\gr{\vpth}\ulambda\gr{t^n_{k+1}-r} \\&\qquad - \lstake{k}\gr{\vpth}\gr[\big]{\vpth\gr{r}-\vpth\gr{\stopt_k}-\llambda\Delta_n} - \lstake{k}\llambda\gr{\vpth}\gr{t^n_{k+1}-r}.
\end{align*}
Recall that \(g^n_k=\gr{\eye+\Delta_n\upoissgen}g^n_{k+1}=g^n_{k+1}+\Delta_n\upoissgen g^n_{k+1}\) by definition.
Hence, if \(\ustake{k}\gr{\vpth}>0\) (and therefore \(\lstake{k}\gr{\vpth}=0\)),
\begin{align*}
\MoveEqLeft
g^n_{k+1}\gr{\vpth\gr{r}}-g^n_k\gr{\vpth\gr{\stopt_k}} \\
&= g^n_{k+1}\gr[\big]{\vpth\gr{r}} - g^n_{k+1}\gr[\big]{\vpth\gr{\stopt_k}} - \ulambda\Delta_n \gr[\Big]{g^n_{k+1}\gr[\big]{\vpth\gr{\stopt_k}+1}-g^n_{k+1}\gr[\big]{\vpth\gr{\stopt_k}}}\\
&= g^n_{k+1}\gr[\big]{\vpth\gr{r}} - g^n_{k+1}\gr[\big]{\vpth\gr{\stopt_k}} - \ulambda\Delta_n \ustake{k}\gr{\vpth};
\end{align*}
a similar equality holds if \(\lstake{k}\gr{\vpth}>0\) (and therefore \(\ustake{k}\gr{\vpth}=0\)) with \(\llambda\) in place of~\(\llambda\).
Because \(r<\sigma_n\gr{\vpth}\) by assumption, it's guaranteed that \(0\leq\vpth\gr{r}-\vpth\gr{\stopt_k}\leq 1\); it's therefore straightforward to verify that if \(\ustake{k}\gr{\vpth}>0\),
\begin{equation*}
\ustake{k}\gr{\vpth}\gr[\big]{\vpth\gr{r}-\vpth\gr{\stopt_k}-\ulambda\Delta_n}
= g^n_{k+1}\gr{\vpth\gr{r}}-g^n_k\gr{\vpth\gr{\stopt_k}},
\end{equation*}
and similarly for \(\lstake{k}\gr{\pth}>0\); this verifies \zcref{eqn:proof of sublinear watanabe:capital change before sigma}.
In the particular case that \(\stopt_{k+1}\gr{\pth}<\sigma_n\gr{\pth}\), \zcref{eqn:proof of sublinear watanabe:capital change before sigma} for \(r=\stopt_{k+1}\gr{\pth}\) reduces to
\begin{equation}\label{eqn:proof of sublinear watanabe:capital change between stopt before sigma}
\capital{\betstrat, c}{\stopt_{k+1}}\gr{\pth}
- \capital{\betstrat , c}{\stopt_k}\gr{\pth}
= g^n_{k+1}\gr{\pth\gr{\stopt_{k+1}}}-g^n_k\gr{\pth\gr{\stopt_{k+1}}}.
\end{equation}

Since \(\capital{\betstrat, c}{\stopt_1}\gr{\vpth}=g^n_1\gr{\vpth\gr{s}}+\delta\) for all \(\vpth\in\tostop{s}{\pth}\), it follows from \zcref{eqn:proof of sublinear watanabe:capital change between stopt before sigma} that for all \(k\in\set{1, \dots, n}\) such that \(\stopt_{k+1}\gr{\vpth}<\sigma_n\gr{\vpth}\),
\begin{equation}\label{eqn:proof of sublinear watanabe:capital on stopt}
\capital{\betstrat, c}{\stopt_k}\gr{\vpth}
= g^n_k\gr[\big]{\vpth\gr{\stopt_k}}+\delta.
\end{equation}
Now recall that \(\sigma_n\gr{\pth}=+\infty\) for any \(\vpth\in A_n\cap \tostop{t_1}{\pth}\), so it follows from the preceding equality for \(k=n+1\) that
\begin{equation*}
\capital{\betstrat, c}{t}\gr{\vpth}
= \capital{\betstrat, c}{\stopt_{n+1}}\gr{\vpth}
= g^n_{n+1}\gr[\big]{\pth\gr{\stopt_{n+1}}}+\delta \\
= g\gr[\big]{\vpth\gr{t}}+\delta.
\end{equation*}
Next, observe that for \(\vpth\in\compl{A_n}\cap\tostop{s}{\pth}\), there is some largest~\(k\in\set{1, \dots, n}\) such that \(\stopt_k\gr{\vpth}<\sigma_n\gr{\vpth}\) (and of course \(\stopt_{k+1}\gr{\vpth}=\sigma_n\gr{\vpth}\)); then since \(\vpth\) has unit jumps,
\begin{equation*}
\lim_{r\nearrow \sigma_n\gr{\vpth}} \capital{\betstrat, c}{\sigma_n}\gr{\vpth}
- \capital{\betstrat, c}{r}\gr{\vpth}
= \lim_{r\nearrow\sigma_n\gr{\vpth}} \diffstake{k}\gr{\vpth}\gr{\vpth\gr{\sigma_n}-\vpth\gr{r}}
= \diffstake{k}\gr{\pth}.
\end{equation*}
It follows from this, \zcref{eqn:proof of sublinear watanabe:capital change before sigma,eqn:proof of sublinear watanabe:capital on stopt} that
\begin{align*}
\capital{\betstrat, c}{\sigma_n}\gr{\vpth}
&= \lim_{r\nearrow\sigma_n\gr{\vpth}}\delta+g^n_{k+1}\gr{\vpth\gr{r}} + \gr[\big]{\ustake{k}\gr{\vpth}\ulambda-\lstake{k}\gr{\vpth}\llambda}\gr{t^n_{k+1}-r} + \diffstake{k}\gr{\vpth} \\
&\geq \delta+\inf g - \gr{\sup g-\inf g}\llambda\Delta_n - \gr{\sup g-\inf g} \\
&= 2\inf g-\sup g.
\end{align*}
Recall now that \(\oscuprev{\indica{\compl{A_n}}}\gr{\pth}<\epsilon/2\gr{\sup g-\inf g}\), so there is some positive capital process~\(\capital{\betstrat', c'}{\bullet}\in\oscapprocs\) such that \(\capital{\betstrat', c'}{s}\gr{\pth}<\epsilon/2\gr{\sup g-\inf g}\) and \(\capital{\betstrat', c'}{t}\gr{\vpth}\geq\indica{\compl{A_n}}\gr{\vpth}\) for all \(\vpth\in\tostop{s}{\pth}\).
Since \(\oscapprocs\) is a cone and \(g\gr{N_t}\) is \(t\)-measurable, we conclude from this and \zcref{lem:stopt-measurable then stop there} that
\begin{align*}
\oscuprev{g\gr{N_t}}[s]\gr{\pth}
&\leq \capital{\betstrat, c}{s}\gr{\pth} + 2\gr{\sup g-\inf g}\capital{\betstrat', c'}{s}\gr{\pth} \\
&< g^n_1\gr{\pth\gr{s}}+\delta+\epsilon \\
&< \br{\upoisssg_{t-s}g}\gr{\pth\gr{s}}+3\epsilon.
\end{align*}
Since \(\pth\in\countpaths\), \(g\in\bfns\) and \(\epsilon\in\posreals\) were arbitrary, this proves the inequality in the statement.
\end{relproof}

\subsection{Proof of \zcref{prop:expected renewal time}}\label{asec:proof of expected renewal time}
Finally, we set out to prove \zcref{prop:expected renewal time}, and we'll do so with the help of the following result.
\begin{lemma}\label{lem:upper probability of not jumping}
For any \(t,\Delta\in\posreals\),
\begin{equation*}
\oscuprev{\indica{\set{N_{t+\Delta}=N_t}}}[t]
= e^{-\Delta\llambda}
\quad\text{with }
\set{N_{t+\Delta}=N_t}\coloneqq\set{\pth\in\countpaths\colon \pth\gr{t+\Delta}=\pth\gr{t}}.
\end{equation*}
\end{lemma}
The reader will have no difficulty in verifying that while our proof invokes results that come after \zcref{prop:expected renewal time}---in particular \zcref{the:with sublinear poisson semigroup}---these don't rely on this particular result.
\begin{proof}
Fix any \(\pth\in\countpaths\), and let \(x\coloneqq\pth\gr{t}\).
Then it follows from the definition of~\(\oscuprev{\bullet}[t]\gr{\pth}\) that
\begin{equation*}
\oscuprev{\indica{\set{N_{t+\Delta}=N_t}}}[t]\gr{\pth}
= \oscuprev{\indica{x}\gr{N_{t+\Delta}}}[t]\gr{\pth}.
\end{equation*}
From this, \zcref{the:with sublinear poisson semigroup} and Lemma~53 in \cite{2019Erreygers}, it follows that
\begin{equation*}
\oscuprev{\indica{\set{N_{t+\Delta}=N_t}}}[t]\gr{\pth}
= \br{\upoisssg_\Delta\indica{x}}\gr{x}
= \br{\upoisssg_\Delta\indica{0}}\gr{0}
= \lim_{k\to+\infty} \gr*{\eye+\frac{\Delta}{k}\upoissgen}^k\indica{0}\gr{0}.
\end{equation*}
To obtain the equality in the statement, it therefore suffices to observe that (i) for any \(k\in\nats\) such that \(1-\llambda\Delta/k\geq0\),
\begin{equation*}
\gr*{\eye+\frac{\Delta}{k}\upoissgen}^\ell\indica{0}
= \gr*{1-\frac{\Delta\llambda}{k}}^\ell\indica{0}
\quad\text{for all } \ell\in\set{0, \dots, k};
\end{equation*}
and (ii) \(\lim_{k\to+\infty} \gr{1-\Delta\llambda/k}^k=e^{-\Delta\llambda}\).
\qed\end{proof}

\begin{relproof}{\zcref{prop:expected renewal time}}
Due to a similar argument as the one at the end of our proof for \zcref{prop:expected increment}, it suffices to prove that
\begin{equation}\label{eq:bounds on sigma_stopt}
\oscuprev{\rho_\stopt-\stopt}
\leq \frac1{\llambda}
\quad\text{and}\quad
\oscuprev{-\gr{\rho_\stopt-\stopt}}
\leq -\frac1{\ulambda}.
\end{equation}

If \(\llambda=0\), the first inequality in~\eqref{eq:bounds on sigma_stopt} is trivial.
If \(\llambda>0\), we let \(c\coloneqq1/\llambda\) and consider the trading strategy~\(\betstrat\) with stopping times~\(\stopt_1\coloneqq\stopt\) and \(\stopt_2\coloneqq\rho_\stopt\) and stakes~\(\ustake{1}\coloneqq0\) and \(\lstake{1}\coloneqq1/\llambda\).
Then for all \(\vpth\in\countpaths\) and \(r\in\posreals\),
\begin{equation*}
\capital{\betstrat, c}{r}\gr{\vpth}
= \begin{cases}
\frac1\llambda &\text{if } r\leq\stopt\gr{\vpth} \\
\frac1\llambda+\gr{r-\stopt\gr{\vpth}} &\text{if } \stopt\gr{\vpth}<r<\rho_\stopt\gr{\vpth} \\
\rho_\stopt\gr{\vpth}-\stopt\gr{\vpth} &\text{if } r\geq\rho_\stopt\gr{\vpth}.
\end{cases}
\end{equation*}
From this equality, we see that \(\capital{\betstrat, c}{\bullet}\) is bounded below (by~\(1/\llambda\)) and superhedges~\(\rho_\stopt-\stopt\).
This implies the first inequality in~\eqref{eq:bounds on sigma_stopt}.

Establishing the second inequality is a bit more involved.
Let \(t\coloneqq\stopt\gr{\pth}\), and note that (i) \(\tostop{\stopt}{\pth}=\tostop{t}{\pth}\); and (ii) \(\rho_\stopt\gr{\vpth}-\stopt\gr{\vpth}=\rho_t\gr{\vpth}-t\) for all \(\vpth\in\tostop{\stopt}{\pth}\).
It follows from this and the definition of~\(\oscuprev{\bullet}[\stopt]\gr{\pth}\) that
\begin{equation}\label{eq:rho_stopt to rho_t}
\oscuprev{-\gr{\rho_\stopt-\stopt}}[\stopt]\gr{\pth}
= \oscuprev{-\gr{\rho_t-t}}[t]\gr{\pth}.
\end{equation}

Let \(c\) be any strictly positive real number such that \(c<1/\ulambda\), fix some \(\Delta\in\posreals\) and consider the trading strategy~\(\betstrat_{c,\Delta}\) with stopping times~\(\stopt_1\coloneqq t\) and \(\stopt_2\coloneqq\rho_t\wedge\gr{t+\Delta}\) and stakes~\(\ustake{1}\coloneqq c\) and \(\lstake{1}\coloneqq0\).
Then for all \(\vpth\in\countpaths\) and \(r\in\posreals\),
\begin{equation*}
\capital{\betstrat_{c,\Delta}, -c}{r}\gr{\vpth}
= \begin{cases}
-c &\text{if } r\leq t \\
-c-c\ulambda\gr{r-t} &\text{if } t<r<\rho_t\gr{\vpth}\wedge\gr{t+\Delta} \\
-c\ulambda\gr{\rho_t\gr{\vpth}-t} &\text{if } r\geq\rho_t\gr{\vpth}\leq t+\Delta \\
-c-c\ulambda\Delta &\text{if } r\geq t+\Delta<\rho_t\gr{\vpth}.
\end{cases}
\end{equation*}
From this equality, we see that \(\capital{\betstrat_{c,\Delta}, -c}{\bullet}\) is bounded below and superhedges \(-\gr{\rho_t-t}\) on~\(\set{\vpth\in\countpaths\colon \rho_t\gr{\vpth}\leq t+\Delta}\).
It remains for us to counteract the additional~\(-c\) term on the remaining set~\(S\coloneqq\set{\vpth\in\countpaths\colon \rho_t\gr{\vpth}>t+\Delta}=\set{\vpth\in\countpaths\colon \vpth\gr{t+\Delta}=\vpth\gr{t}}\).
By \zcref{lem:upper probability of not jumping,lem:capital above remaining infimum}, for all \(\epsilon\in\sposreals\) there is some positive capital process~\(\capital{\betstrat'_{\epsilon, \Delta}, c_\epsilon}{\bullet}\) that superhedges~\(\indica{S}\) on~\(\tostop{t}{\pth}\) with \(\capital{\betstrat'_{\epsilon, \Delta}, c_\epsilon}{t}\gr{\pth}<e^{-\Delta\llambda}+\epsilon\).
Consequently, the capital process
\begin{equation*}
\capital{\betstrat_{c,\Delta}, -c}{\bullet}+c\capital{\betstrat_{\epsilon, \Delta}, c_\epsilon}{\bullet}
\in\oscapprocs
\end{equation*}
is bounded below and superhedges~\(-\gr{\rho_t-t}\) on~\(\tostop{t}{\pth}=\tostop{\stopt}{\pth}\), and therefore
\begin{equation*}
\oscuprev{\rho_t-t}[t]\gr{\pth}
\leq \capital{\betstrat_{c,\Delta}, -c}{t}\gr{\pth}+c\capital{\betstrat_{\epsilon, \Delta}, c_\epsilon}{t}\gr{\pth}
< -c + c\gr{e^{-\Delta\llambda}+\epsilon}.
\end{equation*}
Since this inequality holds for arbitrary \(c<1/\ulambda\), \(\Delta\in\posreals\) and \(\epsilon\in\sposreals\), it follows [taking \(\epsilon\) arbitrarily small, \(\Delta\) arbitrarily large and \(c\) as large as possible] that
\begin{equation*}
\oscuprev{\rho_\stopt-\stopt}\gr{\pth}
= \oscuprev{\rho_t-t}[t]\gr{\pth}
\leq -\frac1{\ulambda},
\end{equation*}
where the first equality is \zcref{eq:rho_stopt to rho_t}.
This proves the second inequality in \eqref{eq:bounds on sigma_stopt}.
\end{relproof}

\fi
\end{document}